\documentclass[12pt]{amsart}

\usepackage[margin=1in]{geometry}

\usepackage[all]{xy}
\usepackage{amssymb}
\usepackage{amscd,latexsym,amsthm,amsfonts,amssymb,amsmath,amsxtra}
\usepackage[mathscr]{eucal}
\usepackage[colorlinks]{hyperref}
\usepackage{mathtools}

\usepackage{bbm}
\usepackage{stmaryrd}

\usepackage{chngcntr}
\numberwithin{equation}{subsection}
\newcounter{keepeqno}
\newenvironment{num}
 {\setcounter{keepeqno}{\value{equation}}%
  \begin{list}{(\theequation)}{\usecounter{equation}}%
  \setcounter{equation}{\value{keepeqno}}}
 {\end{list}}

\hypersetup{linkcolor=black, citecolor=black}

\newcommand{\BC}{{\mathbb {C}}}

\newcommand{\BG}{{\mathbb {G}}}
\newcommand{\BH}{{\mathbb {H}}}

\newcommand{\BR}{{\mathbb {R}}}

\newcommand{\BZ}{{\mathbb {Z}}}

\newcommand{\CC}{{\mathcal {C}}}

\newcommand{\CI}{{\mathcal {I}}}

\newcommand{\CL}{{\mathcal {L}}}

\newcommand{\CO}{{\mathcal {O}}}

\newcommand{\CS}{{\mathcal {S}}}

\newcommand{\Fg}{{\mathfrak {g}}}

\newcommand{\Fi}{{\mathfrak {i}}}

\newcommand{\Fo}{{\mathfrak {o}}}

\newcommand{\Fs}{{\mathfrak {s}}}
\newcommand{\Ft}{{\mathfrak {t}}}

\newcommand{\RB}{{\mathrm {B}}}

\newcommand{\RH}{{\mathrm {H}}}
\newcommand{\RI}{{\mathrm {I}}}

\newcommand{\RM}{{\mathrm {M}}}
\newcommand{\RN}{{\mathrm {N}}}
\newcommand{\RO}{{\mathrm {O}}}

\newcommand{\RS}{{\mathrm {S}}}

\newcommand{\Ad}{{\mathrm{Ad}}}

\newcommand{\disc}{{\mathrm{disc }}}

\newcommand{\Gal}{{\mathrm{Gal}}}
\newcommand{\GL}{{\mathrm{GL}}}

\newcommand{\Hom}{{\mathrm{Hom}}}

\renewcommand{\Im}{{\mathrm{Im}}}
\newcommand{\Ind}{{\mathrm{Ind}}}

\newcommand{\Id}{{\mathrm{Id}}}

\newcommand{\Lie}{{\mathrm{Lie}}}

\newcommand{\PGL}{{\mathrm{PGL}}}

\newcommand{\reg}{{\mathrm{reg}}}

\newcommand{\SO}{{\mathrm{SO}}}

\newcommand{\sgn}{{\mathrm{sgn}}}
\newcommand{\Sp}{{\mathrm{Sp}}}

\newcommand{\tr}{{\mathrm{tr}}}

\newcommand{\ud}{\,\mathrm{d}}

\newcommand{\udl}{\underline}
\newcommand{\wt}{\widetilde}
\newcommand{\wh}{\widehat}

\newcommand{\bs}{\backslash}

\def\alp{{\alpha}}

\def\del{{\delta}}
\def\Del{{\Delta}}

\def\eps{{\epsilon}}

\def\veps{{\varepsilon}}

\def\sig{{\sigma}}

\def\nil{\mathrm{Nil}}
\def\Nr{\mathrm{Nr}}

\def\std{\mathrm{std}}

\def\lam{{\lambda}}

\def\Sig{{\Sigma}}

\def\Gam{{\Gamma}}

\def\LG{{}^{L}G}

\def\vphi{\varphi}

\def\p{\prime}
\def\rss{\mathrm{rss}}
\def\geom{\mathrm{geom}}

\def\ss{\mathrm{ss}}

\def\ad{\mathrm{ad}}

\def\Vg{\mathrm{Vogan}}
\def\rel{\mathrm{rel}}
\def\qs{\mathrm{qs}}
\def\PI{\mathrm{PI}}
\def\NI{\mathrm{NI}}
\def\qd{\mathrm{qd}}
\def\cpt{\mathrm{cpt}}
\def\SO{\mathrm{SO}}
\def\Norm{\mathrm{Norm}}

\newtheorem{thm}{Theorem}[subsection]
\newtheorem{defin}[thm]{Definition}
\newtheorem{rmk}[thm]{Remark}

\newtheorem{pro}[thm]{Proposition}
\newtheorem{lem}[thm]{Lemma}
\newtheorem{cor}[thm]{Corollary}

\newtheorem{conjec}[thm]{Conjecture}

\makeatletter

\newcommand{\Rmnum}[1]{\expandafter\@slowromancap\romannumeral #1@}
\makeatother

\DeclarePairedDelimiter{\floor}{\lfloor}{\rfloor}



\usepackage{xcolor}
\usepackage{needspace}
\newcounter{cmpheading}
\makeatletter
\newcommand{\cmpHeading}[1]{%
  \par\addvspace{\medskipamount}\refstepcounter{cmpheading}%
  \noindent\textbf{#1}\par\smallskip}

\makeatother

\begin{document}

\title[The local Gross-Prasad conjecture over $\BR$: epsilon dichotomy]{The local Gross-Prasad conjecture over $\BR$: epsilon dichotomy}

\author{Cheng Chen}
\address{Department of Mathematics, the University of Hong Kong, Hong Kong
}
\email{chen5968@hku.hk}

\author{Zhilin Luo}
\address{Department of Mathematics\\
Purdue University\\
West Lafayette, IN 47907, USA}
\email{luo642@purdue.edu}

\begin{abstract}
Following the work of Jean-Loup Waldspurger, we prove the epsilon dichotomy part of the local Gross-Prasad conjecture over $\BR$ for tempered local $L$-parameters.
\end{abstract}

\maketitle

\tableofcontents

\section{Introduction}\label{sec:intro}

This paper continues our study of the \emph{local Gross-Prasad conjecture}
\cite{MR1186476,MR1295124}. Let $F$ be a local field of characteristic zero,
and let $(W,V)$ be a pair of non-degenerate quadratic spaces over $F$ such that $
W\subset V$ and $W^\perp$ is split of odd dimension.
Set $
G=\SO(W)\times \SO(V).$
Let $N$ be the unipotent radical of the parabolic subgroup of $\SO(V)$ attached
to a full isotropic flag in $W^\perp$, and put
$
H=\SO(W)\ltimes N.
$
Fix a generic character $\xi$ of $N(F)$ extended to $H(F)$ trivially on $\SO(W)(F)$. For an irreducible admissible representation $\pi$ of
$G(F)$, define
\begin{equation}\label{eq:intro:1}
m(\pi)=\dim \Hom_{H(F)}(\pi,\xi),
\end{equation}
where, in the archimedean case, $\Hom_{H(F)}(\pi,\xi)$ denotes the space of
continuous intertwiners. The multiplicity-one theorem
\cite{agrsmut1,ggporiginal,MR3202559,szmut1,MR2720228} asserts that
\[
m(\pi)\leq 1.
\]

The local Gross-Prasad conjecture speculates a refinement for the behavior of the multiplicity \eqref{eq:intro:1}. The pure inner forms relevant to the pair
$(W,V)$ are indexed by
\[
H^1(F,\SO(W))\simeq H^1(F,H),
\]
which classifies quadratic spaces over $F$ with the same dimension and
discriminant as $W$. For $\alp\in H^1(F,H)$, let $W_\alp$ be the corresponding
quadratic space and set
\[
V_\alp=W_\alp\oplus^\perp W^\perp,\qquad
G_\alp=\SO(W_\alp)\times \SO(V_\alp).
\]
Then $G_\alp$ is a pure inner form of $G$ with ${}^LG_\alp\simeq \LG$. For a local $L$-parameter
\[
\vphi:\CL_F\to \LG,
\]
where $\CL_F$ denotes the local Langlands group of $F$, let
$\Pi^{G_\alp}(\vphi)$ be the local $L$-packet of $G_\alp$
\cite{langlandsproblems}. Following D. Vogan \cite{MR1216197}, define the Vogan $L$-packet attached to
$\vphi$ by
\[
\Pi^{\Vg}(\vphi)=
\bigsqcup_{\alp\in H^1(F,G)}\Pi^{G_\alp}(\vphi).
\]
After the choice of a Whittaker datum for the family of pure inner forms
$\{G_\alp\}_{\alp\in H^1(F,G)}$, Vogan's parametrization gives a
non-degenerate pairing
\[
\Pi^{\Vg}(\vphi)\times \CS_\vphi \to \{\pm 1\},
\]
known over archimedean fields by \cite[Thm.~6.3]{MR1216197}. Here
\[
\CS_\vphi=\pi_0(S_\vphi),\qquad
S_\vphi=Z_{\wh{G}}(\Im(\vphi)).
\]
Thus each $\pi\in \Pi^{\Vg}(\vphi)$ determines a character
\[
\chi_\pi:\CS_\vphi\to \{\pm 1\}.
\]
For the Gross-Prasad conjecture, the relevant packet is obtained by restricting
to the pure inner forms $(G_\alp,H_\alp,\xi_\alp)$ arising from $\alp\in H^1(F,H)\to H^1(F,G)$:
\[
\Pi^{\Vg}_\rel(\vphi)
=
\bigsqcup_{\alp\in H^1(F,H)}\Pi^{G_\alp}(\vphi).
\]
The Gross-Prasad conjecture may then be stated as follows.

\begin{conjec}[\cite{MR1186476,MR1295124}]\label{conjec:ggp:intro}
For every generic local $L$-parameter $\vphi$, the following assertions hold:
\begin{enumerate}
    \item
    \[
    \sum_{\pi\in \Pi^{\Vg}_{\rel}(\vphi)}m(\pi)=1.
    \]

    \item
    Fix the Whittaker datum for $\{G_\alp\}_{\alp\in H^1(F,G)}$ as in
    \cite[(6.3)]{MR1295124}. Let $\pi_\vphi$ be the unique representation in
    $\Pi^\Vg_\rel(\vphi)$ satisfying $m(\pi_\vphi)=1$. Then
    \[
    \chi_{\pi_\vphi}=\chi_\vphi,
    \]
    where $\chi_\vphi$ is defined in \eqref{eq:chivphicharacter:1}.
\end{enumerate}
\end{conjec}

In the $p$-adic case, Conjecture \ref{conjec:ggp:intro} for tempered local
$L$-parameters was proved by Waldspurger
\cite{waldspurger10,MR3155344,MR3155345,waldspurgertemperedggp,MR3202559}.
The extension to generic local $L$-parameters was obtained by C.~M\oe glin and
Waldspurger from the tempered case \cite{MR3155346}.

For archimedean $F$, Part (1) of Conjecture \ref{conjec:ggp:intro} for tempered
local $L$-parameters was proved over $\BR$ by the second author in his thesis
\cite{thesis_zhilin}, following the strategy of Waldspurger and
Beuzart-Plessis \cite{beuzart2015local}. In \cite{chen2021local}, the first
author reduced Conjecture \ref{conjec:ggp:intro} for generic local
$L$-parameters over $\BR$ to the tempered case, following M\oe glin--Waldspurger
and Xue \cite{xue1besselgeneric}. Consequently, Conjecture
\ref{conjec:ggp:intro} over $\BC$ was established in \cite{chen2021local}.
Part (1) over $\BC$ had also been proved in \cite{mollers2017symmetry} in the
codimension-one case.

It remains to prove Part (2) over $\BR$ for tempered parameters. This is the
main result of the present paper.

\begin{thm}\label{thm:toprove}
Over $\BR$, Part (2) of Conjecture \ref{conjec:ggp:intro} holds for tempered
local $L$-parameters.
\end{thm}

The proof follows the strategy of Waldspurger \cite{MR3155345}, with
two archimedean modifications. We first recall the $p$-adic mechanism. By Fourier inversion on the finite abelian group $\CS_\vphi$, it is enough to prove, for every
$s\in \CS_\vphi$, the identity
\begin{equation}\label{eq:intro:2}
\sum_{\pi\in \Pi^\Vg_{\rel}(\vphi)}
\chi_{\pi}(s)m(\pi)
=
\chi_{\vphi}(s).
\end{equation}
Waldspurger's geometric multiplicity formula
\cite{waldspurger10,waldspurgertemperedggp} gives
\[
m(\pi)=m_\geom(\pi)
\]
for tempered representations, where $m_\geom(\pi)$ is an integral expression involving
the germ expansions of the distribution character of $\pi$. Thus
\eqref{eq:intro:2} is equivalent to
\[
\sum_{\pi\in \Pi^\Vg_{\rel}(\vphi)}
\chi_{\pi}(s)m_\geom(\pi)
=
\chi_{\vphi}(s).
\]
For each $s\in \CS_\vphi$, ordinary endoscopy attaches to $G$ an endoscopic
group
\[
G_{1,s}\times G_{2,s},
\]
where each $G_{i,s}$ is again of Gross-Prasad type, attached to a pair of
quadratic spaces $(W_i,V_i)$. The corresponding endoscopic character identity
gives
\[
\sum_{\pi\in \Pi^\Vg_{\rel}(\vphi)}
\chi_{\pi}(s)m_\geom(\pi)
=
m^S_\geom(\vphi_1)m^S_\geom(\vphi_2),
\]
as established in \cite[Prop.~3.3]{MR3155345}. Here $m^S_\geom$ denotes the
stable geometric multiplicity introduced in \cite[\S 3.2]{MR3155345}. In the
$p$-adic proof, twisted endoscopy and the twisted character identity
\cite[\S 1.6,\S 1.8]{MR3155345} identify this stable geometric term with the
geometric multiplicity for a pair of twisted general linear groups of
Gross-Prasad type; the latter is computed in \cite{MR3155344} and gives
$\chi_\vphi(s)$.

Over $\BR$, the argument closes without twisted endoscopy. Using
\cite[Thm.~4.4]{moeglin2020paquets}\cite[\S 6]{renard2024generic}, we first reduce the proof of Theorem \ref{thm:toprove}
to parameters of good parity; equivalently, to parameters having only
$\mathrm{O}$-type irreducible constituents, as recalled in
Subsection \ref{subsec:conjecofggp}. Further parabolic reductions preserve the component group.
Proper elliptic endoscopic decompositions then give an induction on the
total dimension of the two orthogonal groups. The only nontrivial
remaining case is the classical situation for
$\SO(2)\times\SO(3)$.

Two archimedean points require separate treatment. First, over $p$-adic fields
, there are only two relevant pairs
$(W_\alp,V_\alp)_{\alp\in H^1(F,H)}$, while over $\BR$, the set of pure inner
forms is larger. We organize the relevant families by the Kottwitz sign of
$G_\alp$ \cite{kottsign}. Second, Waldspurger's proof of
\cite[Lem.~13.4~(ii)]{waldspurger10}, an essential input in the geometric
multiplicity argument, uses \cite[Conj.~1.2]{waldtransfert}. The corresponding
archimedean statement was established only recently by the authors in
\cite{chen2025fourier}. In the present paper, the required form is obtained
from a formula of Rossmann \cite{MR508985}. These two points are treated in
Sections \ref{sec:regconj} and \ref{sec:geomultiplicity}.

We finally recall the parallel unitary and Fourier--Jacobi settings. Gan,
Gross and Prasad formulated analogous conjectures for unitary groups in
\cite{ggporiginal}. In the Bessel case, Beuzart-Plessis proved the tempered
$p$-adic unitary case following Waldspurger
\cite{raphaelpadic,MR3251763}, while Gan and Ichino proved the generic
$p$-adic case \cite{MR3573972} following M\oe glin--Waldspurger.
Over $\BR$, Beuzart-Plessis proved Part (1) for tempered unitary parameters
\cite{beuzart2015local}; Xue proved Part (2) in the tempered case
\cite{xue1bessel} and reduced the generic case to the tempered case
\cite{xue1besselgeneric}. The Fourier--Jacobi cases of \cite{ggporiginal}
involve Weil representations and concern skew-hermitian unitary groups and
symplectic-metaplectic groups. Over $p$-adic fields, these cases were resolved
by Gan--Ichino for skew-hermitian unitary groups \cite{MR3573972}, and by Atobe
for symplectic-metaplectic groups \cite{MR3788848}, using theta-correspondence
methods that reduce the Fourier--Jacobi case to the Bessel case.

\subsubsection{Organization}

Throughout the paper we work over $\BR$.

In Section \ref{sec:localgp} we fix notation and conventions and recall the local Gross-Prasad conjecture.

Sections \ref{sec:regconj} and \ref{sec:geomultiplicity} establish the
archimedean inputs needed for the endoscopic reduction.

\begin{itemize}
    \item
    In Section \ref{sec:regconj}, we review the parametrization of regular
    semisimple conjugacy classes in special orthogonal groups, following
    \cite[\S 1.3]{MR2672539} and \cite[\S 1.3, \S 1.4]{MR3155345}. We then
    prove Proposition \ref{pro:unionpureinnerconjclassof}, the archimedean
    analogue of the description of the fibers of $p_G$ in
    \cite[\S 1.4]{MR3155345}. This proposition describes the union of regular
    semisimple conjugacy classes over pure inner forms with fixed Kottwitz
    sign.

    \item
    In Section \ref{sec:geomultiplicity}, we recall the geometric multiplicity
    formula of \cite{thesis_zhilin}. We then prove
    Lemma \ref{lem:thecalculationofj}, the archimedean analogue of
    \cite[Lem.~13.4~(ii)]{waldspurger10}. As a consequence, we obtain
    Corollary \ref{cor:germfunaslimitofquasichar}, corresponding to
    \cite[\S 13.6]{waldspurger10}, which expresses the germs of a distribution
    character in terms of the distribution character itself. We also recall the
    stable geometric multiplicity introduced in \cite[\S 3.2]{MR3155345}, and
    prove Lemma \ref{lem:fiberoverkappappunionoverpure}, describing the union
    of the supports of the geometric multiplicity formula over Gross-Prasad
    triples with fixed Kottwitz sign.
\end{itemize}

In Section \ref{sec:redtoendoscopypf}, we carry out the reduction of
Theorem \ref{thm:toprove}. Parabolic reduction first reduces to the good-parity
case in the sense of \cite[\S 4.1]{moeglin2020paquets}. We then use parabolic reductions preserving the component groups and
proper elliptic endoscopic reduction to reach the classical
$\SO(2)\times\SO(3)$ case.

\medskip
\noindent\textbf{AI disclosure.}
During the revision of the article, the authors use ChatGPT 6 Astra to discover some minor mistakes in the article.

\medskip
\noindent\textbf{Acknowledgments.}
We thank D. Jiang for suggesting the problem and for helpful comments. We thank
C. Wan for helpful discussions and suggestions concerning the
endoscopic reduction. The authors thank R. Chen for helpful information
on theta correspondence used in the previous proof of the basic cases, and C. Lo for
helpful discussions on the local Langlands correspondence. This work was
supported in part by a Research Assistantship from NSF grant DMS-1901802. C. Chen
was supported by the European Union's Horizon 2020 research and innovation programme under the Marie Skłodowska-Curie grant agreement No. 101034255, and he thanks the University of Hong Kong for accommodation. 
Z. Luo thanks the Department of Mathematics at the University of Chicago for
research support through the Dickson Instructorship, and Purdue University for
startup support. We
thank the anonymous referee for various critical and helpful comments.

\section{Local Gross-Prasad conjecture}\label{sec:localgp}

We recall the local Gross-Prasad conjecture following \cite{MR1186476,MR1295124}.

\subsection{Gross-Prasad triples}\label{subsec:gptriples}

We recall the notion of Gross-Prasad triples
\cite[\S 6]{thesis_zhilin}.

Let $(W,V)$ be a pair of non-degenerate quadratic spaces over $\BR$. The pair
$(W,V)$ is called \textbf{admissible} if there exist an anisotropic line
$D=\BR z_0$ and a split non-degenerate quadratic space $Z$ of dimension $2r$
over $\BR$ such that
\[
V\simeq W\oplus^\perp D\oplus^\perp Z.
\]
Let $q$ denote the quadratic form on $V$. Choose a basis $\{z_i\}_{i=\pm 1}^{\pm r}$ of $Z$ such that
\[
q(z_i,z_j)=\del_{i,-j},\qquad i,j\in \{\pm 1,\ldots,\pm r\}.
\]
Let $N$ be the unipotent radical of the parabolic subgroup of $\SO(V)$
stabilizing the totally isotropic flag
\[
\langle z_r\rangle
\subset
\langle z_r,z_{r-1}\rangle
\subset \cdots \subset
\langle z_r,\ldots,z_1\rangle .
\]
Set
\[
G=\SO(W)\times \SO(V).
\]
We regard $\SO(W)$ as an algebraic subgroup of $G$ by the diagonal embedding;
hence $\SO(W)$ acts on $N$ by conjugation. Put
\[
H=\SO(W)\ltimes N.
\]
Define a morphism $\lam:N\to \BG_a$ by
\[
\lam(n)=\sum_{i=0}^{r-1}q(z_{-i-1},nz_i),\qquad n\in N.
\]
The morphism $\lam$ is invariant under conjugation by $\SO(W)$, and therefore
extends uniquely to a morphism on $H$ which is trivial on $\SO(W)$; we denote
this extension again by $\lam$. Let
\[
\lam_\BR:H(\BR)\to \BR
\]
be the induced map on real points. Fix a nontrivial unitary character $\psi$ of
$\BR$, and define
\[
\xi(h)=\psi(\lam_\BR(h)),\qquad h\in H(\BR).
\]

\begin{defin}\label{defin:ggptriple}
With the notation above, $(G,H,\xi)$ is called the \textbf{Gross-Prasad triple}
attached to the admissible pair $(W,V)$.
\end{defin}

\subsection{\texorpdfstring{Vogan $L$-packets}{Vogan L-packets}}\label{subsec:vgLpacket}

We recall the Vogan $L$-packets of special orthogonal groups over $\BR$,
following \cite[\S 3]{MR1186476} and \cite{MR1216197}.

Let $\CL_\BR$ be the Weil group of $\BR$. By the local Langlands correspondence
over $\BR$ \cite{langlandsclassify}, for a reductive algebraic group $G$ over
$\BR$, local $L$-parameters for $G$ determine finite local $L$-packets of
irreducible Casselman--Wallach representations of $G(\BR)$
\cite{MR1013462,MR1170566}. 

A local $L$-parameter is a $\wh{G}$-conjugacy class
of admissible homomorphisms
\[
\vphi:\CL_\BR\to \LG
\]
whose image consists of semisimple elements. Here $\wh{G}$ is the dual group of
$G$, and $\LG$ is its Langlands dual group. The parameter $\vphi$ is called
\textbf{tempered} if its image is bounded.

Pure inner forms of $G$ share the same (Langlands) dual groups as $G$. Thus, a local $L$-parameter for $G$ is also a local $L$-parameter for every pure
inner form of $G$. Following Vogan \cite{MR1216197}, one therefore works not
with a single local $L$-packet but with the Vogan $L$-packet
\[
\bigsqcup_{G^\p}\Pi^{G^\p}(\vphi),
\]
where $G^\p$ runs over the isomorphism classes of pure inner forms of $G$.

We now specialize to special orthogonal groups. Let $(V,q)$ be a
non-degenerate quadratic space over $\BR$. The pure inner forms of $\SO(V)$ are
classified by
\[
H^1(\BR,\SO(V)),
\]
which identifies with the set of quadratic spaces over $\BR$ having the same
dimension and discriminant as $V$ \cite[\S 8]{MR1295124}. Quadratic spaces over
$\BR$ are classified by their signatures $(p,q)$, where
\[
p=\PI(V),\qquad q=\NI(V).
\]
For a quadratic space of signature $(p,q)$, the discriminant is
\begin{equation}\label{eq:discriminantdef}
\disc(V)=(-1)^{\floor{\frac{\dim V}{2}}}(-1)^q
\in \{\pm 1\}\simeq \BR^\times/\BR^{\times 2}.
\end{equation}
Thus the pure inner forms of $\SO(p,q)$ are precisely
\begin{equation}\label{eq:pureinnerofSOpq}
\SO(p_\alp,q_\alp),
\qquad
p_\alp+q_\alp=p+q,\qquad
p_\alp\equiv p\pmod{2}.
\end{equation}

Among the pure inner forms of $\SO(V)$, the quasi-split and split forms are as
follows.
\begin{num}
\item\label{conjec:pureinner:1}
If $\dim V$ is odd, or if $\dim V$ is even and
\[
\PI(V)-\NI(V)\equiv 0\pmod{4},
\]
then $\SO(V)$ has a unique quasi-split pure inner form, and this form is split
over $\BR$.

\item\label{conjec:pureinner:2}
If $\dim V$ is even and
\[
\PI(V)-\NI(V)\equiv 2\pmod{4},
\]
then $\SO(V)$ has two quasi-split pure inner forms,
\[
\SO(p+2,p),\qquad \SO(p,p+2),
\qquad p=\frac{\dim V}{2}-1.
\]
These two groups are isomorphic as inner forms, but not as pure inner forms.
\end{num}

Two admissible pairs $(W,V)$ and $(W^\p,V^\p)$ are called
\textbf{relevant} if
\begin{align*}
\dim W=\dim W^\p,\qquad
\disc(W)=\disc(W^\p),\qquad
\dim V=\dim V^\p,\qquad
\disc(V)=\disc(V^\p).
\end{align*}
Fix an admissible pair $(W,V)$ with associated Gross-Prasad triple
$(G,H,\xi)$. For each $\alp\in H^1(\BR,\SO(W))$, let $W_\alp$ be the
corresponding quadratic space and set
\[
V_\alp=W_\alp\oplus^\perp W^\perp.
\]
This gives a unique admissible pair $(W_\alp,V_\alp)$ relevant to $(W,V)$, with
associated Gross-Prasad triple $(G_\alp,H_\alp,\xi_\alp)$. For a local
$L$-parameter
\[
\vphi:\CL_\BR\to \LG,
\]
define the \textbf{relevant Vogan $L$-packet} by
\begin{equation}\label{eq:relvgpacket}
\Pi^{\Vg}_\rel(\vphi)
=
\bigsqcup_{\alp\in H^1(\BR,\SO(W))}
\Pi^{G_\alp}(\vphi).
\end{equation}

\begin{rmk}\label{rmk:uniqueqsggp}
For an admissible pair $(W,V)$, it may occur that $\SO(W)$ or $\SO(V)$ has two
quasi-split pure inner forms in the sense of \eqref{conjec:pureinner:2}. Since
$\dim V-\dim W$ is odd, the relevance condition nevertheless singles out a
unique admissible pair $(W_\qs,V_\qs)$ relevant to $(W,V)$ such that
$\SO(W_\qs)\times \SO(V_\qs)$ is quasi-split.
\end{rmk}

\subsection{The conjecture}\label{subsec:conjecofggp}

We recall the distinguished character of Gross-Prasad
\cite[\S 10]{MR1186476} and the corresponding local conjecture
\cite{MR1186476,MR1295124}.

Fix a Whittaker datum for the family of pure inner forms
$\{G_\alp\}_{\alp\in H^1(F,G)}$, namely a quasi-split pure inner form of $G$
, together with a Borel subgroup over $\BR$ and a generic character of the
unipotent radical. A local $L$-parameter $\vphi$ is called \textbf{generic} if
the Vogan $L$-packet $\Pi^{\Vg}(\vphi)$ contains a generic representation.

We use the following form of Vogan's parametrization \cite[Thm.~6.3]{MR1216197}.

\begin{num}
\item\label{num:shelstadvogandual}
Let $G$ be a reductive algebraic group over $\BR$, and let
\[
\vphi:\CL_\BR\to \LG
\]
be a local $L$-parameter. Put
\[
S_\vphi=Z_{\wh{G}}(\Im(\vphi)),\qquad
\CS_\vphi=\pi_0(S_\vphi).
\]
After fixing a Whittaker datum for $G$, every generic $L$-parameter $\vphi$
admits a bijection between $\Pi^{\Vg}(\vphi)$ and the group of characters of
$\CS_\vphi$, normalized so that the trivial character corresponds to the
generic representation.
\end{num}

We now describe the component group for special orthogonal groups. Let $V$ be a
non-degenerate quadratic space over $\BR$, and let $\vphi_V$ be a local
$L$-parameter of $\SO(V)$. We denote by $\mathrm{std}$ the standard embedding of
${}^L\SO(V)$ into $\GL(\RM_V)$, which is the $2n$-dimensional representation of
\[
{}^L\SO(V)=\begin{cases}
    \mathrm{Sp}_{2n}(\mathbb C)&\text{ when }\dim V=2n+1,\\
    \mathrm{SO}_{2n}(\mathbb C)&\text{ when }\dim V=2n\text{ and $V$ has a split pure inner form,}\\
    \mathrm{O}_{2n}(\BC)&\text{ otherwise. }
\end{cases}
\] Composing $\vphi_V$ with $\std$ gives a representation
\[
\mathrm{std}\circ \vphi_V:\CL_\BR\to \GL(\RM_V).
 \]

The space $\RM_V$ carries a non-degenerate $\CL_{\BR}$-invariant bilinear form $B$ of sign $\eps\in\{\pm 1\}$. If $\dim V$ is odd, then $\RB$ is symplectic and
$\eps=-1$; if $\dim V$ is even, then $\RB$ is symmetric and $\eps=1$.

By semisimplicity,
\[
\RM_V=\bigoplus_i m_i\RM_{i,V},
\]
where the $\RM_{i,V}$ are pairwise non-isomorphic irreducible representations
of $\CL_\BR$. Following \cite[Prop.~6.5, Prop.~7.6]{MR1186476}, the irreducible
summands fall into the following three types:
\begin{num}
\item[]
\begin{enumerate}
\item[($\RO$-type)]
$\RM_{i,V}\simeq \RM_{i,V}^\vee$, and $\RM_{i,V}$ carries a non-degenerate
$\CL_\BR$-invariant pairing of sign $\eps$;

\item[($\Sp$-type)]
$\RM_{i,V}\simeq \RM_{i,V}^\vee$, and $\RM_{i,V}$ carries a non-degenerate
$\CL_\BR$-invariant pairing of sign $-\eps$. In this case $m_i$ is even;

\item[($\GL$-type)]
$\RM_{i,V}\not\simeq \RM_{i,V}^\vee$. Then
$\RM_{i,V}^\vee\simeq \RM_{j,V}$ for some $j\neq i$, and $m_i=m_j$.
\end{enumerate}
\end{num}

Let $\RI_\RO$ and $\RI_\Sp$ be the index sets of the $\RO$-type and
$\Sp$-type summands, respectively. Let $\RI_\GL$ index the unordered pairs
$\RM_{i,V}\oplus \RM_{i,V}^\vee$ of $\GL$-type. By
\cite[Prop.~6.6, Prop.~7.7]{MR1186476}, the centralizer of the image of
$\vphi_V$ is
\begin{equation}\label{eq:SvphiVdescription}
S_{\vphi_V}
=
\bigg(
\prod_{i\in \RI_\RO}\RO(m_i,\BC)
\bigg)_+
\times
\prod_{i\in \RI_\Sp}\Sp(m_i,\BC)
\times
\prod_{i\in \RI_\GL}\GL(m_i,\BC),
\end{equation}
where
\[
\bigg(
\prod_{i\in \RI_\RO}\RO(m_i,\BC)
\bigg)_+
=
S\bigg(
\prod_{\substack{i\in \RI_\RO\\ \dim \RM_{i,V}\ \mathrm{odd}}}
\RO(m_i,\BC)
\bigg)
\times
\prod_{\substack{i\in \RI_\RO\\ \dim \RM_{i,V}\ \mathrm{even}}}
\RO(m_i,\BC).
\]
Here
\[
S\bigg(
\prod_{\substack{i\in \RI_\RO\\ \dim \RM_{i,V}\ \mathrm{odd}}}
\RO(m_i,\BC)
\bigg)
\]
denotes the subgroup on which the product of determinants is equal to $1$.
Consequently,
\begin{equation}\label{eq:component}
\CS_{\vphi_V}
\simeq
\begin{cases}
(\BZ/2\BZ)^r,
& \text{if every $\RO$-type $\RM_{i,V}$ has even dimension},\\
(\BZ/2\BZ)^{r-1},
& \text{otherwise},
\end{cases}
\end{equation}
where $r=|\RI_\RO|$.

We now recall the distinguished character defined by Gross and Prasad in
\cite[\S 10]{MR1186476}. Let $(G,H,\xi)$ be the Gross-Prasad triple attached
to an admissible pair $(W,V)$ over $\BR$, and let
\[
\vphi=\vphi_W\times \vphi_V
\]
be a local $L$-parameter of
\[
G=\SO(W)\times \SO(V).
\]
After fixing the Whittaker datum for $G$, Vogan's parametrization gives a
pairing
\[
\Pi^{\Vg}(\vphi)\times \CS_\vphi\to \{\pm 1\}.
\]
Thus every $\pi\in \Pi^{\Vg}(\vphi)$ determines a character
\[
\chi_\pi:\CS_\vphi\to \{\pm 1\}.
\]
Since
$
\CS_\vphi=\CS_{\vphi_W}\times \CS_{\vphi_V},
$
Gross and Prasad define a character
\[
\chi_\vphi=\chi^V_{\vphi_W}\times \chi^W_{\vphi_V}
\]
of $\CS_\vphi$ as follows. For
$
s=s_W\times s_V\in \CS_{\vphi_W}\times \CS_{\vphi_V},
$
set
\begin{align}\label{eq:chivphicharacter:1}
\chi^W_{\vphi_V}(s_V)
&=
\det(\RM_V^{s_V=-1})^{\frac{\dim \RM_W}{2}}(-1)
\det(\RM_W)^{\frac{\dim \RM_V^{s_V=-1}}{2}}(-1)
\veps\bigg(\frac{1}{2},\RM_V^{s_V=-1}\otimes \RM_W,\psi\bigg),
\\
\chi^V_{\vphi_W}(s_W)
&=
\det(\RM_W^{s_W=-1})^{\frac{\dim \RM_V}{2}}(-1)
\det(\RM_V)^{\frac{\dim \RM_W^{s_W=-1}}{2}}(-1)
\veps\bigg(\frac{1}{2},\RM_W^{s_W=-1}\otimes \RM_V,\psi\bigg).
\nonumber
\end{align}
Here $\RM_V^{s_V=-1}$ and $\RM_W^{s_W=-1}$ denote the $(-1)$-eigenspaces of
$s_V$ and $s_W$, respectively; $\det \RM_V$ and $\det \RM_W$ denote determinant
characters of Weil group representations; and $\veps(\cdots)$ denotes the
corresponding local root number. The $\epsilon$-factors appearing in \ref{eq:chivphicharacter:1} are independent of the choice of the nontrivial additive character $\psi$ (\cite[\S 18]{ggporiginal}).

We now state the local Gross-Prasad conjecture. Let $\pi$ be an irreducible
Casselman--Wallach representation of $G(\BR)$ and set
\begin{equation}\label{eq:multiplicity}
m(\pi)=\dim \Hom_{H(\BR)}(\pi,\xi).
\end{equation}
By \cite{szmut1,MR2720228},
\[
m(\pi)\leq 1.
\]
The conjecture determines which member of the relevant Vogan $L$-packet has
nonzero multiplicity.

\begin{conjec}\label{conjec:gp}
Let $(G,H,\xi)$ be the Gross-Prasad triple attached to an admissible pair
$(W,V)$ over $\BR$. Fix a generic local $L$-parameter $\vphi$ of $G$. Then:

\begin{enumerate}
\item
There exists a unique representation
\[
\pi_\vphi\in \Pi^\Vg_{\rel}(\vphi)
\]
such that
\[
m(\pi_\vphi)=1.
\]

\item
Fix the Whittaker datum for $G$ as in \cite[(6.3)]{MR1295124}. Under the
parametrization in \ref{num:shelstadvogandual}, the character attached to
$\pi_\vphi$ satisfies
\[
\chi_{\pi_\vphi}=\chi_\vphi,
\]
where $\chi_\vphi$ is defined in \eqref{eq:chivphicharacter:1}.
\end{enumerate}
\end{conjec}

For tempered local $L$-parameters, Part (1) of Conjecture \ref{conjec:gp} was
proved by the second author in \cite{thesis_zhilin}, following Waldspurger
\cite{waldspurger10,MR3155345} and Beuzart-Plessis
\cite{beuzart2015local}. Following \cite{MR3155346}, the first author reduced
Conjecture \ref{conjec:gp} for generic local $L$-parameters over $\BR$ to the
tempered case in \cite{chen2021local}. Thus the remaining assertion over
$\BR$ is Part (2) for tempered parameters. This is the main theorem of the
paper.

\begin{thm}\label{thm:main}
Over $\BR$, Part (2) of Conjecture \ref{conjec:gp} holds for tempered local
$L$-parameters.
\end{thm}

\section{Some regular semisimple conjugacy classes}\label{sec:regconj}

In this section we recall the parametrization of certain regular semisimple conjugacy classes in special orthogonal groups, following
\cite[\S 1.3]{MR2672539} and \cite[\S 1.3, \S 1.4]{MR3155345}. The Lie
algebra analogue appears in \cite{wald01nilpotent}; see also
\cite[\S 5.1]{thesis_zhilin}. In Subsection \ref{subsec:unionoverpure} we
prove Proposition \ref{pro:unionpureinnerconjclassof}, which describes the
union of these parametrizations over pure inner forms with fixed Kottwitz sign.

\subsection{Parametrization}\label{subsec:parametrization}

We first recall the parametrization of the relevant regular semisimple
conjugacy classes.

\begin{num}
\item\label{num:parameterization:datum}
Consider the following data:
\begin{itemize}
\item a finite set $I$;

\item for each $i\in I$, a finite extension $F_{\pm i}$ of $\BR$ and a
quadratic étale $F_{\pm i}$-algebra $F_i$; denote by $\tau_i$ the nontrivial
automorphism of $F_i$ over $F_{\pm i}$;

\item for each $i\in I$, an element $u_i\in F_i^\times$ satisfying
\[
u_i\tau_i(u_i)=1.
\]
\end{itemize}
\end{num}

Let $\udl{\Xi}$ be the set of quadruples
\[
\kappa=(I,(F_{\pm i})_{i\in I},(F_i)_{i\in I},(u_i)_{i\in I})
\]
satisfying \ref{num:parameterization:datum}. Two quadruples
\[
\kappa=(I,(F_{\pm i})_{i\in I},(F_i)_{i\in I},(u_i)_{i\in I})
\quad \text{and}\quad 
\kappa^\p=(I^\p,(F^\p_{\pm i})_{i\in I^\p},(F^\p_i)_{i\in I^\p},
(u^\p_i)_{i\in I^\p})
\]
are called isomorphic if there exist a bijection $\iota:I\to I^\p$ and
compatible isomorphisms
\[
\iota_{\pm i}:F_{\pm i}\to F^\p_{\pm\iota(i)},\qquad
\iota_i:F_i\to F^\p_{\iota(i)}
\]
such that
\[
\iota_i(u_i)=u^\p_{\iota(i)}
\]
for every $i\in I$.

A quadruple $\kappa\in \udl{\Xi}$ is called \textbf{regular} if its only
automorphism is the identity. Let $\Xi_\reg$ be the set of isomorphism classes
of regular quadruples. For an even positive integer $d$, let $\Xi_{\reg,d}$ be
the subset of $\Xi_\reg$ consisting of classes
\[
\kappa=(I,(F_{\pm i})_{i\in I},(F_i)_{i\in I},(u_i)_{i\in I})
\]
such that
\[
\sum_{i\in I}[F_i:\BR]=d.
\]

For $\kappa\in \Xi_\reg$, let $I^*=I^*_\kappa$ be the subset of $I$ consisting
of those $i$ for which $F_i$ is a field. Equivalently, $i\in I^*$ precisely
when the quadratic character
\[
\sgn_{F_i/F_{\pm i}}:F_{\pm i}^\times\to \{\pm 1\}
\]
is nontrivial. Define
\[
C(\kappa)
=
\prod_{i\in I}
F^\times_{\pm i}/\Norm_{F_i/F_{\pm i}}(F_i^\times)
\simeq
\prod_{i\in I^*}\{\pm 1\}.
\]
For $c=(c_i)_{i\in I}\in C(\kappa)$, define a quadratic space
$(W_{\kappa,c},q_{\kappa,c})$ by
\[
W_{\kappa,c}=\bigoplus_{i\in I}F_i
\]
and
\begin{equation}\label{eq:quadkappac:1}
q_{\kappa,c}
\bigg(
\sum_{i\in I}w_i,
\sum_{i\in I}w_i^\p
\bigg)
=
\sum_{i\in I}
\tr_{F_i/\BR}\big(\tau_i(w_i)w_i^\p c_i\big),
\qquad
w_i,w_i^\p\in F_i.
\end{equation}
Here each $c_i$ is represented by an element of $F_{\pm i}^\times$. Since
$\tau_i(c_i)=c_i$, the form \eqref{eq:quadkappac:1} is symmetric. Its
isomorphism class is independent of the representatives of the classes $c_i$
\cite[\S 1.3]{MR2672539}.

The signature of $(W_{\kappa,c},q_{\kappa,c})$ is computed explicitly as
follows.

\begin{lem}\label{lem:positivenegativeindex}
Let
\[
\RI^\pm_{\BC}
=
\{i\in I\mid F_i\simeq \BC,\ c_i=\pm 1\},
\]
and let
\[
\RI_{\BR\oplus \BR}
=
\{i\in I\mid F_i\simeq \BR\oplus \BR\},
\qquad
\RI_{\BC\oplus \BC}
=
\{i\in I\mid F_i\simeq \BC\oplus \BC\}.
\]
Then
\begin{align*}
\PI(W_{\kappa,c})
&=
2|\RI^+_\BC|
+
|\RI_{\BR\oplus \BR}|
+
2|\RI_{\BC\oplus \BC}|,\\
\NI(W_{\kappa,c})
&=
2|\RI^-_\BC|
+
|\RI_{\BR\oplus \BR}|
+
2|\RI_{\BC\oplus \BC}|.
\end{align*}
\end{lem}

\begin{proof}
There are three cases.

If $F_{\pm i}=\BR$ and $F_i=\BC$, then the summand is
\[
\tr_{\BC/\BR}(\tau_i(w_i)w_i^\p c_i).
\]
It is positive definite for $c_i=1$ and negative definite for $c_i=-1$.

If $F_{\pm i}=\BR$ and $F_i=\BR\oplus \BR$, then
\[
\tr_{F_i/\BR}\big(\tau_i(w_i^1,w_i^2)(w_i^{\p 1},w_i^{\p 2})\big)
=
w_i^1w_i^{\p 2}+w_i^2w_i^{\p 1},
\]
which has signature $(1,1)$.

If $F_{\pm i}=\BC$ and $F_i=\BC\oplus \BC$, then
\[
\tr_{F_i/\BR}
\big(\tau_i(w_i^1,w_i^2)(w_i^{\p 1},w_i^{\p 2})c_i\big)
=
c_i\tr_{\BC/\BR}
\big(w_i^1w_i^{\p 2}+w_i^2w_i^{\p 1}\big),
\]
which has signature $(2,2)$. Summing over $i\in I$ gives the assertion.
\end{proof}

\begin{rmk}\label{rmk:changeindextosign}
Let
\[
\kappa=(I,(F_{\pm i})_{i\in I},(F_i)_{i\in I},(u_i)_{i\in I})
\in \Xi_{\reg,d}.
\]
By Lemma \ref{lem:positivenegativeindex}, the isomorphism class of
$(W_{\kappa,c},q_{\kappa,c})$ is determined by the cardinalities of
$\RI_\BC^+$ and $\RI_\BC^-$. Since
\[
\RI_\BC^+\sqcup \RI_\BC^-=I^*,
\]
these cardinalities are equivalently determined by
\[
\sum_{i\in I^*}c_i=|\RI_\BC^+|-|\RI_\BC^-|.
\]
For
\[
\theta\in \{-|I^*|,-|I^*|+2,\ldots,|I^*|-2,|I^*|\},
\]
set
\[
C(\kappa)_\theta
=
\left\{
c=(c_i)\in C(\kappa)
\ \middle|\
\sum_{i\in I^*}c_i=\theta
\right\}.
\]
Then
\[
(W_{\kappa,c},q_{\kappa,c})
\simeq
(W_{\kappa,c^\p},q_{\kappa,c^\p})
\]
if and only if $c$ and $c^\p$ lie in the same $C(\kappa)_\theta$. In that case
\begin{equation}\label{eq:indexofWkappac}
\PI(W_{\kappa,c})=\frac{d}{2}+\theta,
\qquad
\NI(W_{\kappa,c})=\frac{d}{2}-\theta.
\end{equation}
\end{rmk}

Define $x_{\kappa,c}\in \GL(W_{\kappa,c})$ by
\begin{equation}\label{eq:quadkappac:2}
x_{\kappa,c}
\bigg(
\sum_{i\in I}w_i
\bigg)
=
\sum_{i\in I}u_iw_i,
\qquad w_i\in F_i.
\end{equation}
The relation $u_i\tau_i(u_i)=1$ implies
\[
x_{\kappa,c}\in \SO(W_{\kappa,c}).
\]

\begin{defin}\label{defin:parconj}
Let $(V,q)$ be a non-degenerate quadratic space over $\BR$, and put
\[
\Del_V=\PI(V)-\NI(V).
\]

\begin{enumerate}
\item
If $\dim V$ is even, define
\[
\Xi_{\reg,V}
=
\left\{
(\kappa,c)
\ \middle|\
\kappa\in \Xi_{\reg,\dim V},\quad
c\in C(\kappa)_{\frac{\Del_V}{2}}
\right\}.
\]

\item
If $\dim V$ is odd, define $\Xi_{\reg,V}$ to be the set of pairs
$(\kappa,c)$ such that
\[
\kappa\in \Xi_{\reg,\dim V-1},\qquad c\in C(\kappa),
\]
and such that there exists an anisotropic line $(D_{\kappa,V},q_{\kappa,V})$
with
\[
(W_{\kappa,c},q_{\kappa,c})
\oplus^\perp
(D_{\kappa,V},q_{\kappa,V})
\simeq
(V,q).
\]
\end{enumerate}
\end{defin}

\begin{rmk}\label{rmk:afterdefparconj}
Assume $\dim V$ is odd. Comparing discriminants shows that the signature of the
anisotropic line $D_{\kappa,V}$ is independent of $c\in C(\kappa)$. We denote
this sign by
\begin{equation}\label{eq:definoffraki}
\Fi_{V,\kappa}
=
(-1)^{
\frac{-\Del_V+1}{2}+|I_\kappa^*|
}
\in \{\pm 1\}.
\end{equation}
Thus $\Fi_{V,\kappa}$ depends only on $\kappa$ and the pure inner class of
$V$.

Comparing signatures gives
\[
\sum_{i\in I^*_\kappa}c_i
=
\frac{\Del_V-\Fi_{V,\kappa}}{2}.
\]
Equivalently,
\begin{equation}\label{eq:xiregValternative}
\Xi_{\reg,V}
=
\left\{
(\kappa,c)
\ \middle|\
\kappa\in \Xi_{\reg,\dim V-1},\quad
c\in C(\kappa)_{\frac{\Del_V-\Fi_{V,\kappa}}{2}}
\right\}.
\end{equation}
\end{rmk}

The following parametrization is the real form of the parametrization in
\cite[\S 1.3]{MR2672539} and \cite[\S 1.3, \S 1.4]{MR3155345}; it also follows
by the same argument as in \cite[\S 5.1]{thesis_zhilin}.

\begin{thm}\label{thm:parconjso}
Let $(V,q)$ be a non-degenerate quadratic space over $\BR$. Let
$\SO(V)^\rss/\sim$ be the set of regular semisimple conjugacy classes in
$\SO(V)$. If $\dim V$ is even, let
$\SO(V)^\rss_{\neq \pm 1}/\sim$ denote the subset consisting of classes without
eigenvalue $\pm 1$. If $\dim V$ is odd, let $\RS\RO(V)^\rss_{\neq -1}/\sim$ denote the subset of classes without eigenvalue $-1$.

\begin{enumerate}
\item
If $\dim V$ is even, there is a two-to-one map
\[
\SO(V)^\rss_{\neq \pm 1}/\sim
\longrightarrow
\Xi_{\reg,V}.
\]
More precisely, let $(\kappa,c)\in \Xi_{\reg,V}$. Then
\[
(W_{\kappa,c},q_{\kappa,c})\simeq (V,q)
\]
by \eqref{eq:indexofWkappac}. The element $x_{\kappa,c}$ has no eigenvalue
$\pm 1$ by regularity of $\kappa$. Its $\mathrm{O}(V)$-conjugacy class in
$\SO(V)$ splits into two distinct $\SO(V)$-conjugacy classes, denoted
\[
x_{\kappa,c}^+,\qquad x_{\kappa,c}^-.
\]

\item
If $\dim V$ is odd, there is a bijection
\[
\Xi_{\reg,V}
\longleftrightarrow
\SO(V)^\rss_{\neq -1}/\sim.
\]
For $(\kappa,c)\in \Xi_{\reg,V}$, choose an isomorphism
\[
(W_{\kappa,c},q_{\kappa,c})
\oplus^\perp
(D_{\kappa,V},q_{\kappa,V})
\simeq
(V,q).
\]
Then
\[
x_{\kappa,c}^{D_{\kappa,V}}
=
\Id_{D_{\kappa,V}}\oplus x_{\kappa,c}
\]
defines a single $\SO(V)$-conjugacy class in $\SO(V)$.
\end{enumerate}
\end{thm}

\subsection{Union over pure inner forms with fixed Kottwitz sign}
\label{subsec:unionoverpure}

We now take the union of the parametrizations in Theorem \ref{thm:parconjso}
over pure inner forms with fixed Kottwitz sign. The result is Proposition
\ref{pro:unionpureinnerconjclassof}, the archimedean analogue of the
description of the fibers of $p_G$ in \cite[\S 1.4]{MR3155345}.

We first recall the Kottwitz sign \cite{kottsign}.

\begin{defin}\label{defin:kottwitzsign}
Let $G$ be a reductive algebraic group over $\BR$, and let $K$ be a maximal
compact subgroup of $G(\BR)$. Let $G_\qs$ be the quasi-split inner form of
$G$, and let $K_\qs$ be a maximal compact subgroup of $G_\qs(\BR)$. The
Kottwitz sign of $G$ is
\[
e(G)=(-1)^{\frac{\dim K_\qs-\dim K}{2}}.
\]
\end{defin}

For special orthogonal groups this sign is explicit.

\begin{lem}\label{lem:kottwitzsignsopq}
For $\SO(p,q)$ one has
\[
e(\SO(p,q))
=
\begin{cases}
1, & p+q \text{ even},\\
(-1)^{\frac{(p-q)^2-1}{8}}, & p+q \text{ odd}.
\end{cases}
\]
\end{lem}

Combining this formula with \eqref{eq:pureinnerofSOpq} gives the following
criterion.

\begin{cor}\label{cor:pureinnerformsopq}
Let $\alp\in H^1(\BR,\SO(V))$, and let $V_\alp$ have signature
$(p_\alp,q_\alp)$. Suppose $V$ has signature $(p,q)$.

\begin{enumerate}
\item
If $p+q$ is odd, then
\[
e(\SO(V))=e(\SO(V_\alp))
\Longleftrightarrow
p\equiv p_\alp \pmod{4}.
\]

\item
If $p+q$ is even, then for any fixed anisotropic line $D$,
\[
e(\SO(V\oplus^\perp D))
=
e(\SO(V_\alp\oplus^\perp D))
\Longleftrightarrow
p\equiv p_\alp \pmod{4}.
\]
\end{enumerate}
\end{cor}

We now state the main result of this subsection. For $\kappa\in \Xi_{\reg,d}$,
write
\[
C(\kappa)^{\pm 1}
=
\left\{
c\in C(\kappa)
\ \middle|\
\prod_{i\in I^*_\kappa}c_i=\pm 1
\right\}.
\]

\begin{pro}\label{pro:unionpureinnerconjclassof}
Let $V$ be a non-degenerate quadratic space over $\BR$, and let
$e_0\in \{\pm 1\}$.

\begin{enumerate}
\item
Assume $\dim V$ is odd. Then
\[
\bigsqcup_{\substack{
\alp\in H^1(\BR,\SO(V))\\
e(\SO(V_\alp))=e_0
}}
\Xi_{\reg,V_\alp}
=
\Xi_{\reg,\dim V-1,e_0},
\]
where
\[
\Xi_{\reg,\dim V-1,e_0}
=
\left\{
(\kappa,c)
\ \middle|\
\kappa\in \Xi_{\reg,\dim V-1},\quad
c\in C(\kappa)^{e_0\eps_{V,\kappa}}
\right\},
\]
and
\[
\eps_{V,\kappa}
=
(-1)^{
\frac{
-\PI(V_\qs)
+
|I^*_\kappa|
+
\frac{\dim V+\Fi_{V,\kappa}}{2}
}{2}
}.
\]
Here $V_\qs$ denotes the unique quasi-split pure inner form of $V$.

\item
Assume $\dim V$ is even, and fix an anisotropic line $D$ with
\[
\mathrm{sig}(D)\in \{\pm 1\}.
\]
Then
\[
\bigsqcup_{\substack{
\alp\in H^1(\BR,\SO(V))\\
e(\SO(V_\alp\oplus^\perp D))=e_0
}}
\Xi_{\reg,V_\alp}
=
\Xi_{\reg,\dim V,e_0,D},
\]
where
\[
\Xi_{\reg,\dim V,e_0,D}
=
\left\{
(\kappa,c)
\ \middle|\
\kappa\in \Xi_{\reg,\dim V},\quad
c\in C(\kappa)^{e_0\eps_{V,\kappa,D}}
\right\},
\]
and
\[
\eps_{V,\kappa,D}
=
(-1)^{
\frac{
|I^*_\kappa|
+
\frac{\dim V+1+\mathrm{sig}(D)}{2}
-
\PI(V,D)
}{2}
}.
\]
Here $\PI(V,D)$ is the positive index of the unique quasi-split pure inner form
of $\SO(V\oplus^\perp D)$.
\end{enumerate}
\end{pro}

\begin{proof}
Assume first that $\dim V$ is odd. Fix
\[
\alp\in H^1(\BR,\SO(V))
\]
with
\[
e(\SO(V_\alp))=e_0.
\]
Let $(\kappa,c)\in \Xi_{\reg,V_\alp}$. By
\eqref{eq:xiregValternative},
\[
\sum_{i\in I^*_\kappa}c_i
=
\frac{\Del_{V_\alp}-\Fi_{V_\alp,\kappa}}{2}.
\]
Since $\Fi_{V_\alp,\kappa}=\Fi_{V,\kappa}$ depends only on the pure inner class,
we write it as $\Fi_{V,\kappa}$.

For any $c\in C(\kappa)$,
\begin{equation}\label{eq:proofunion:1}
\prod_{i\in I^*_\kappa}c_i
=
(-1)^{
\frac{|I^*_\kappa|-\sum_{i\in I^*_\kappa}c_i}{2}
}.
\end{equation}
Therefore, for $(\kappa,c)\in \Xi_{\reg,V_\alp}$,
\[
\prod_{i\in I^*_\kappa}c_i
=
(-1)^{
\frac{
|I^*_\kappa|-
\frac{\Del_{V_\alp}-\Fi_{V,\kappa}}{2}
}{2}
}.
\]
Let $(p_\alp,q_\alp)$ be the signature of $V_\alp$, and let
$(p_\qs,q_\qs)$ be the signature of the quasi-split pure inner form
$V_\qs$. By Corollary \ref{cor:pureinnerformsopq},
\[
e_0=e(\SO(V_\alp))=(-1)^{\frac{p_\alp-p_\qs}{2}}.
\]
Hence
\begin{align*}
e_0\prod_{i\in I^*_\kappa}c_i
&=
(-1)^{
\frac{
p_\alp-p_\qs
+
|I^*_\kappa|
-
\frac{\Del_{V_\alp}-\Fi_{V,\kappa}}{2}
}{2}
}  \\
&=
(-1)^{
\frac{
-p_\qs
+
|I^*_\kappa|
+
\frac{\dim V+\Fi_{V,\kappa}}{2}
}{2}
}
=
\eps_{V,\kappa}.
\end{align*}
Thus
\[
c\in C(\kappa)^{e_0\eps_{V,\kappa}},
\]
and hence
\[
\bigsqcup_{\substack{
\alp\in H^1(\BR,\SO(V))\\
e(\SO(V_\alp))=e_0
}}
\Xi_{\reg,V_\alp}
\subset
\Xi_{\reg,\dim V-1,e_0}.
\]

Conversely, let
\[
(\kappa,c)\in \Xi_{\reg,\dim V-1,e_0}.
\]
The condition
\[
c\in C(\kappa)^{e_0\eps_{V,\kappa}}
\]
together with \eqref{eq:proofunion:1} implies
\[
\sum_{i\in I^*_\kappa}c_i
\equiv
\frac{\Del_{V_\alp}-\Fi_{V,\kappa}}{2}
\pmod{4}
\]
for any pure inner form $V_\alp$ with $e(\SO(V_\alp))=e_0$. Equivalently,
the integer
\[
\Fi_{V,\kappa}+2\sum_{i\in I^*_\kappa}c_i
\]
has the same congruence class modulo $8$ as $\Del_{V_\alp}$ for such
$V_\alp$. Moreover,
\[
\left|
\Fi_{V,\kappa}+2\sum_{i\in I^*_\kappa}c_i
\right|
\leq
1+2|I^*_\kappa|
\leq
\dim V.
\]
As $\alp$ ranges over the pure inner forms of $\SO(V)$ with fixed Kottwitz sign
$e_0$, the integers $\Del_{V_\alp}$ are precisely the integers in
$[-\dim V,\dim V]$ lying in the corresponding congruence class modulo $8$.
Hence there exists $\alp_0$ with
\[
e(\SO(V_{\alp_0}))=e_0
\]
such that
\[
\Del_{V_{\alp_0}}
=
\Fi_{V,\kappa}
+
2\sum_{i\in I^*_\kappa}c_i.
\]
Equivalently,
\[
\sum_{i\in I^*_\kappa}c_i
=
\frac{\Del_{V_{\alp_0}}-\Fi_{V,\kappa}}{2}.
\]
By \eqref{eq:xiregValternative},
\[
(\kappa,c)\in \Xi_{\reg,V_{\alp_0}}.
\]
This proves the reverse inclusion, and hence Part (1).

Assume next that $\dim V$ is even. Fix an anisotropic line $D$, and let
\[
\alp\in H^1(\BR,\SO(V))
\]
satisfy
\[
e(\SO(V_\alp\oplus^\perp D))=e_0.
\]
For $(\kappa,c)\in \Xi_{\reg,V_\alp}$ one has
\[
\sum_{i\in I^*_\kappa}c_i=\frac{\Del_{V_\alp}}{2}.
\]
Thus \eqref{eq:proofunion:1} gives
\[
\prod_{i\in I^*_\kappa}c_i
=
(-1)^{
\frac{
|I^*_\kappa|-\frac{\Del_{V_\alp}}{2}
}{2}
}.
\]
Let
\[
\SO(V_\alp\oplus^\perp D)=\SO(p_\alp,q_\alp),
\]
and let $\SO(p_\qs,q_\qs)$ be its quasi-split pure inner form. By
Corollary \ref{cor:pureinnerformsopq},
\[
e_0=(-1)^{\frac{p_\alp-p_\qs}{2}}.
\]
Since
\[
\Del_{V_\alp}=p_\alp-q_\alp-\mathrm{sig}(D),
\]
we obtain
\begin{align*}
e_0\prod_{i\in I^*_\kappa}c_i
&=
(-1)^{
\frac{
|I^*_\kappa|
-\frac{\Del_{V_\alp}}{2}
+p_\alp-p_\qs
}{2}
} \\
&=
(-1)^{
\frac{
|I^*_\kappa|
+
\frac{\dim V+1+\mathrm{sig}(D)}{2}
-
p_\qs
}{2}
}
=
\eps_{V,\kappa,D}.
\end{align*}
Therefore
\[
\bigsqcup_{\substack{
\alp\in H^1(\BR,\SO(V))\\
e(\SO(V_\alp\oplus^\perp D))=e_0
}}
\Xi_{\reg,V_\alp}
\subset
\Xi_{\reg,\dim V,e_0,D}.
\]

The reverse inclusion is identical to the argument in the odd-dimensional case,
with $\Fi_{V,\kappa}$ replaced by $\mathrm{sig}(D)$ and
$\Del_{V_\alp}$ by the signature difference of $V_\alp$. This proves Part (2).
\end{proof}

\begin{rmk}\label{rmk:afterpro:unionpureinnerconjclassof}
By Remark \ref{rmk:afterdefparconj}, the sign $\Fi_{V,\kappa}$ depends only on
$\kappa$ and the pure inner class of $V$. Hence $\eps_{V,\kappa}$ in
Proposition \ref{pro:unionpureinnerconjclassof}(1), respectively
$\eps_{V,\kappa,D}$ in Proposition \ref{pro:unionpureinnerconjclassof}(2),
depends only on $\kappa$ and on the Kottwitz sign of $V$, respectively of
$V\oplus^\perp D$.
\end{rmk}

\section{Geometric multiplicity formula}\label{sec:geomultiplicity}

Let $(G,H,\xi)$ be the Gross-Prasad triple attached to an admissible pair
$(W,V)$ over $\BR$. In this section we recall the geometric multiplicity
formula for tempered representations of $G(\BR)$ proved in
\cite{thesis_zhilin}. We also establish Lemma
\ref{lem:thecalculationofj} and Corollary
\ref{cor:germfunaslimitofquasichar}, the archimedean analogues of
\cite[Lem.~13.4~(ii)]{waldspurger10} and of the formula in
\cite[\S 13.6]{waldspurger10}. In Subsection \ref{subsec:stablevariant} we
recall the stable geometric multiplicity introduced in \cite[\S 3.2]{MR3155345}.

\subsection{The formula}\label{subsec:theformula}

We recall the geometric multiplicity formula of \cite[\S 7.3]{thesis_zhilin}.

\subsubsection{Geometric support}

Let $H_\ss(\BR)$ be the set of semisimple elements of $H(\BR)$. Every element
of $H_\ss(\BR)$ is $H(\BR)$-conjugate to an element of $\SO(W)_\ss(\BR)$.
Let $\Gam(H)$ denote the set of semisimple conjugacy classes in $H(\BR)$.

For $x\in \SO(W)_\ss(\BR)$, set
\[
W^\p_x=\ker(1-x|_W),\qquad
V^\p_x=\ker(1-x|_V),\qquad
W^{\p\p}_x=\Im(1-x|_W).
\]
Then
\[
W=W^\p_x\oplus W^{\p\p}_x,
\qquad
V=V^\p_x\oplus W^{\p\p}_x,
\]
and $(W^\p_x,V^\p_x)$ is again an admissible pair. Let $G_x$ be the identity
component of the centralizer of $x$ in $G$. Following
\cite[\S 7.3.1]{thesis_zhilin}, one has
\begin{equation}\label{eq:formulaGxdecomp}
G_x=G^\p_x\times G^{\p\p}_x,
\qquad
G^\p_x=\SO(W^\p_x)\times \SO(V^\p_x),
\qquad
G^{\p\p}_x=\SO(W^{\p\p}_x)_x\times \SO(W^{\p\p}_x)_x .
\end{equation}
Let $\Gam(G,H)$ be the subset of $\Gam(H)$ consisting of those semisimple
classes $x$ for which $\SO(W^{\p\p}_x)_x$ is an anisotropic torus and $G_x$ is
quasi-split. This set is endowed with the topology and measure defined in
\cite[(7.3.3)]{thesis_zhilin}.

\subsubsection{The germ $c_\Theta$}

For the Fourier transforms in this section, we use the additive character
$\psi$ used to define $\xi$ and a non-degenerate invariant symmetric
bilinear form $B$ on the Lie algebra, negative on the compact part and
positive on its orthogonal complement. On an orthogonal factor we take
$B(X,Y)=\frac12\operatorname{tr}(XY)$, and on a product we take the sum of
these forms, as in \cite[\S 6.1.3]{thesis_zhilin}. Its restriction to each
semisimple centralizer identifies the Lie algebra with its dual. The Lie
algebra measures are self-dual for $\psi\circ B$, and the nilpotent orbital
measures are the symplectic self-dual measures associated to
\[
(\overline A,\overline C)\longmapsto B(A,[N,C])
\quad\text{on }\Fg_x(\BR)/\Fg_{x,N}(\BR),\qquad N\in\CO,
\]
as in \cite[\S\S 1.6, 1.8]{beuzart2015local}. Thus our Fourier convention is
$\widehat f(Y)=\int f(Z)\psi(B(Y,Z))\ud Z$.

Let $G$ be a reductive algebraic group over $\BR$, and let $\Theta$ be a
quasi-character on $G(\BR)$ in the sense of \cite[\S 4.4]{beuzart2015local}.
For $x\in G_\ss(\BR)$ and for $Y\in \Fg_x(\BR)$ regular semisimple and
sufficiently close to $0$, \cite[Prop.~4.4~(vi)]{beuzart2015local} gives
\[
D^G(xe^Y)^{1/2}\Theta(xe^Y)
=
D^G(xe^Y)^{1/2}
\sum_{\CO\in \nil_\reg(\Fg_x)}
c_{\Theta,\CO}(x)\widehat{j}(\CO,Y)
+
O(|Y|).
\]
Here $\nil_\reg(\Fg_x)$ denotes the set of regular nilpotent orbits in
$\Fg_x(\BR)$, and $\widehat{j}(\CO,\cdot)$ is the Fourier transform of the
nilpotent orbital integral attached to $\CO$.

Let $(V,q)$ be a non-degenerate quadratic space over $\BR$. We call $(V,q)$
\textbf{quasi-split} if
\begin{equation}\label{eq:qsquadraticsp}
(V,q)
\simeq
\BH^{n-1}\oplus^\perp
\begin{cases}
(D,q), & \dim V\equiv 1\pmod 2,\\
(E=F(\sqrt b),c\cdot \RN_{E/F}), & \dim V\equiv 0\pmod 2,
\end{cases}
\end{equation}
for some $b,c\in \BR^\times$. Here $\BH^{n-1}$ is the split quadratic space of
dimension $2n-2$, $(D,q)$ is an anisotropic line, and
\[
(E=F(\sqrt b),c\cdot \RN_{E/F})
\]
denotes the two-dimensional quadratic space
\[
m+n\sqrt b\longmapsto c(m^2-bn^2).
\]

The regular nilpotent orbits in $\Fs\Fo(V)(\BR)$ are as follows
\cite[\S 6.1.2]{thesis_zhilin}:
\begin{enumerate}
\item $\nil_\reg(\Fs\Fo(V))\neq \emptyset$ if and only if $(V,q)$ is
quasi-split;

\item $\nil_\reg(\Fs\Fo(V))$ has one element if $\dim V$ is odd, if
$\dim V\leq 2$, or if $\dim V\geq 4$ is even and $(V,q)$ is quasi-split but not
split;

\item if $(V,q)$ is split of even dimension $\geq 4$, then
$\nil_\reg(\Fs\Fo(V))$ has two elements, denoted $\CO_+$ and $\CO_-$. 
For $\nu\in\{\pm1\}$ and $\dim V=2n$, write
\[
V\simeq\BH^{n-1}\oplus^\perp D_\nu\oplus^\perp D_{-\nu},
\qquad q|_{D_\nu}(u)=\nu u^2.
\]
Then $\CO_\nu$ is the $\SO(V)(\BR)$-orbit of the extension by zero on
$D_{-\nu}$ of a regular nilpotent element of
$\Fs\Fo(\BH^{n-1}\oplus^\perp D_\nu)(\BR)$; this is the indexing of
\cite[\S 6.1.2]{thesis_zhilin}.
\end{enumerate}

Return to the Gross-Prasad triple $(G,H,\xi)$. For $x\in \Gam(G,H)$, the
definition of $\Gam(G,H)$ and \eqref{eq:formulaGxdecomp} give
\[
\nil_\reg(\Fg_x)=\nil_\reg(\Fg^\p_x),
\qquad
\Fg^\p_x=\Fs\Fo(W^\p_x)\times \Fs\Fo(V^\p_x).
\]
Following \cite[\S 7.3.2]{thesis_zhilin}, define $c_\Theta(x)$ as follows.
\begin{enumerate}
\item If $\nil_\reg(\Fg_x)$ has one element $\CO_\reg$, set
\[
c_\Theta(x)=c_{\Theta,\CO_\reg}(x).
\]

\item 
Suppose $\nil_\reg(\Fg_x)$ has two elements. The even-dimensional
factor among $\SO(W^\p_x)$ and $\SO(V^\p_x)$ is split and has dimension at
least $4$, while the odd-dimensional factor is split because it is
quasi-split. Since $x$ acts trivially on the fixed spaces $D$ and $Z$,
\[
V^\p_x=W^\p_x\oplus^\perp D\oplus^\perp Z.
\]
In particular, $\mathrm{sig}(D)$ is independent of $x$.

If $\dim V^\p_x$ is even and $\geq4$, let $\CO_{W^\p_x}$ be the unique
regular nilpotent orbit in $\Fs\Fo(W^\p_x)(\BR)$ and set
\[
c_\Theta(x)=
c_{\Theta,\CO_{W^\p_x}\times\CO_{\mathrm{sig}(D)}}(x),
\]
where $\CO_{\mathrm{sig}(D)}\subset\Fs\Fo(V^\p_x)(\BR)$ is labeled as above.

If $\dim W^\p_x$ is even and $\geq4$, let $\CO_{V^\p_x}$ be the unique
regular nilpotent orbit in $\Fs\Fo(V^\p_x)(\BR)$ and set
\[
c_\Theta(x)=
c_{\Theta,\CO_{-\mathrm{sig}(D)}\times\CO_{V^\p_x}}(x),
\]
where now $\CO_{-\mathrm{sig}(D)}\subset\Fs\Fo(W^\p_x)(\BR)$.
In both formulas the orbit in $\Fg_x$ has zero component in the toral
factor, as in the preceding identification.
\end{enumerate}

\subsubsection{Integral formula}

For $x\in \Gam(G,H)$ set
\[
\Del(x)=|\det(1-x)|_{W^{\p\p}_x}|,
\qquad
D^G(x)=|\det(1-\Ad(x))|_{\Fg/\Fg_x}|.
\]
For a quasi-character $\Theta$ on $G(\BR)$ define
\begin{equation}\label{eq:geomultiplicity}
m_\geom(\Theta)
=
\int_{\Gam(G,H)}
D^G(x)^{1/2}c_\Theta(x)\Del(x)^{-1/2}\,dx .
\end{equation}
The integral is absolutely convergent by
\cite[Prop.~7.3.3.3]{thesis_zhilin}.

For a tempered representation $\pi$ of $G(\BR)$, let $\Theta_\pi$ be its
distribution character \cite{MR145006}, and set
\[
m_\geom(\pi)=m_\geom(\Theta_\pi).
\]
The geometric multiplicity formula of \cite{thesis_zhilin} is the following.

\begin{thm}\label{thm:luothesis}
For every tempered representation $\pi$ of $G(\BR)$,
\[
m_\geom(\pi)=m(\pi).
\]
\end{thm}

\subsection{\texorpdfstring{Special values of $\widehat{j}(\CO,\cdot)$}{Special values of j-hat}}
\label{subsec:specialvalueregularnilporbitalintegral}

We prove the archimedean analogue of \cite[Lem.~13.4~(ii)]{waldspurger10}.
The result will be used to recover the germ $c_\Theta(x)$ from the original
quasi-character $\Theta$.

Let $(V,q)$ be split of even dimension $\geq 4$. Then
$\nil_\reg(\Fs\Fo(V))=\{\CO_+,\CO_-\}$. We use the following family of regular
semisimple elements in $\Fs\Fo(V)(\BR)$.

\begin{num}
\item\label{num:specialvalue:1}
Fix $a_1,a_2\in i\BR^\times$ with $a_1\neq \pm a_2$. Fix an isomorphism of
split quadratic spaces
\[
(V,q)
\simeq
(\BC,c\cdot \Nr)\oplus^\perp(\BC,-c\cdot \Nr)\oplus^\perp(\widetilde Z,q),
\]
where $c=\pm 1$, $\Nr=\Norm_{\BC/\BR}$, and $\widetilde Z$ is split of
dimension $\dim V-4$. Let $\widetilde T$ be a maximal split torus of
$\SO(\widetilde Z)$ with Lie algebra $\widetilde{\Ft}$, and choose a regular
semisimple element $\widetilde S\in \widetilde{\Ft}(\BR)$.

Let
\[
X_{a_1,a_2,\widetilde S}\in \Fs\Fo(V)(\BR)
\]
act on $(\BC,c\cdot \Nr)$ by $a_1$, on $(\BC,-c\cdot \Nr)$ by $a_2$, and on
$\widetilde Z$ by $\widetilde S$. By
\cite[Lem.~5.1.0.5]{thesis_zhilin}, the conjugacy classes inside the stable
class of $X_{a_1,a_2,\widetilde S}$ are parametrized by
$c\in \BR^\times/\BR^{\times 2}$. Let
\[
X^\pm_{a_1,a_2,\widetilde S}
\]
be the two representatives corresponding to $c=c^\pm$, where
\[
\sgn_{\BC/\BR}(c^\pm)
=
\pm \sgn_{\BC/\BR}(\Nr(a_1)-\Nr(a_2)).
\]
\end{num}

Let $\widehat{j}(X^\pm_{a_1,a_2,\widetilde S},\cdot)$ be the Fourier transform
of the Lie algebra orbital integral at
$X^\pm_{a_1,a_2,\widetilde S}$, normalized as in
\cite[\S 1.9]{beuzart2015local}. By \cite[Lem.~4.3.1]{beuzart2015local},
for $Y\in \Fs\Fo(V)(\BR)$ regular semisimple, one has
\begin{align}\label{eq:limitforwhjXpm:1}
&\lim_{\substack{t\in \BR^{\times 2}\\ t\to 0^+}}
D^{\SO(V)}(tY)^{1/2}
\widehat{j}(X^\pm_{a_1,a_2,\widetilde S},tY)
\nonumber\\
&\qquad =
D^{\SO(V)}(Y)^{1/2}
\sum_{\CO\in \nil_\reg(\Fs\Fo(V))}
\Gamma_\CO(X^\pm_{a_1,a_2,\widetilde S})
\widehat{j}(\CO,Y).
\end{align}

Here $\Gamma_\CO$ denote the regular Shalika germs; its explicit value in this
setting is computed in \cite[Thm.~4.2.0.1]{thesis_zhilin}.
We shall use the following germ computation.

\begin{lem}\label{lem:germformula}
\begin{enumerate}
\item For $\CO_\pm\in \nil_\reg(\Fs\Fo(V))$, indexed as in
\cite[\S 6.1.2]{thesis_zhilin},
\[
\Gamma_{\CO_\pm}(X^+_{a_1,a_2,\widetilde S})
-
\Gamma_{\CO_\pm}(X^-_{a_1,a_2,\widetilde S})
=
\pm 1.
\]

\item If $X_\qd\in \Ft_V(\BR)$ is regular semisimple, where $\Ft_V$ is the Lie
algebra of a maximal split torus in $\SO(V)$, then
\[
\Gamma_{\CO_\pm}(X_\qd)=1.
\]
\end{enumerate}
\end{lem}

The required special value is the following.

\begin{lem}\label{lem:thecalculationofj}
With the notation above, for $\nu=\pm 1$,
\[
\widehat{j}(\CO_\nu,X^+_{a_1,a_2,\widetilde S})
=
-\widehat{j}(\CO_\nu,X^-_{a_1,a_2,\widetilde S})
=
-\nu\cdot
|W_{T_\cpt}|
D^{\SO(V)}(X^+_{a_1,a_2,\widetilde S})^{-1/2}.
\]
Here $T_\cpt$ is the centralizer of $X^+_{a_1,a_2,\widetilde S}$ in
$\SO(V)$, and
\[
W_{T_\cpt}=N_{\SO(V)}(T_\cpt)/T_\cpt .
\]
\end{lem}

\begin{proof}
Subtracting the two identities \eqref{eq:limitforwhjXpm:1} and using Lemma
\ref{lem:germformula} gives
\begin{align*}
&\lim_{\substack{t\in \BR^{\times 2}\\ t\to 0^+}}
D^{\SO(V)}(tY)^{1/2}
\left(
\widehat{j}(X^+_{a_1,a_2,\widetilde S},tY)
-
\widehat{j}(X^-_{a_1,a_2,\widetilde S},tY)
\right)
\\
&\qquad =
D^{\SO(V)}(Y)^{1/2}
\left(
\widehat{j}(\CO_+,Y)-\widehat{j}(\CO_-,Y)
\right).
\end{align*}
By \cite[(3.4.6)]{beuzart2015local}, since
$X^\pm_{a_1,a_2,\widetilde S}$ are not split,
\[
\widehat{j}(\CO_+,X^\pm_{a_1,a_2,\widetilde S})
-
\widehat{j}(\CO_-,X^\pm_{a_1,a_2,\widetilde S})
=
2\widehat{j}(\CO_+,X^\pm_{a_1,a_2,\widetilde S}).
\]
It remains to compute the corresponding limit at
$Y=X^\pm_{a_1,a_2,\widetilde S}$.

By the parabolic induction formula for Fourier transforms of orbital integrals
\cite[(3.4.4)]{beuzart2015local}, and by the same reduction as in the last
part of the proof of \cite[Lem.~13.4]{waldspurger10}, the computation reduces
to the case
\[
\dim V=4,\qquad \widetilde Z=0.
\]
Thus it suffices to evaluate
\begin{align}\label{eq:lemmaformula:1}
\lim_{\substack{t\in \BR^{\times 2}\\ t\to 0^+}}
t^{\frac{\dim \SO(V)-\dim T_\cpt}{2}}
\left(
\widehat{j}(X^+_{a_1,a_2},tX^\pm_{a_1,a_2})
-
\widehat{j}(X^-_{a_1,a_2},tX^\pm_{a_1,a_2})
\right).
\end{align}

We now apply Rossmann's formula \cite[p.~217, (15)]{MR508985}. As a real Lie
group,
\[
T_\cpt(\BR)\simeq \SO(2)(\BR)\times \SO(2)(\BR)
\]
is a maximal compact Cartan subgroup of $\SO(2,2)(\BR)$ containing
$X^\pm_{a_1,a_2}$. Moreover
\[
W_{T_\cpt}
=
N_{\SO(2,2)(\BR)}(T_\cpt(\BR))/T_\cpt(\BR)
\simeq \BZ/2\BZ.
\]
Let $W_{T_{\cpt,\BC}}$ be the complex Weyl group attached to
$(\Ft_{\cpt,\BC},\Fs\Fo(4,\BC))$. Then
\[
W_{T_{\cpt,\BC}}\simeq \BZ/2\BZ\times \BZ/2\BZ,
\]
and $W_{T_\cpt}$ embeds diagonally. Fix a positive root system and let
\[
\pi:\Ft_{\cpt,\BC}\to \BC
\]
be the product of the positive roots. 
Rossmann's orbital measure in \cite[(1),~p.218-219]{MR508985} is $|W_{T_{\mathrm{cpt}}}|^{-1}$ times the quotient measure used in \cite[\S1.6,\S 1.8]{beuzart2015local}. Hence the orbital integral normalization in
\cite[(1)]{MR508985} differs from the normalization used here, following
\cite[\S 1.8]{beuzart2015local}, by the factor 
$
\frac{\pi(\cdot)}{|\pi(\cdot)||W_{\mathrm{cpt}}|}.
$
Rossmann's formula gives, for $*= \pm$,
\[
\frac{\pi(X^*_{a_1,a_2})}{|\pi(X^*_{a_1,a_2})|}
\,
\widehat{j}(X^*_{a_1,a_2},tX^\pm_{a_1,a_2})
=
-
\pi(tX^\pm_{a_1,a_2})^{-1}
\sum_{w\in W_{T_\cpt}}
(-1)^{\ell(w)}
e^{i(wtX^\pm_{a_1,a_2},X^*_{a_1,a_2})}.
\]
Taking the limit gives
\[
\frac{\pi(X^*_{a_1,a_2})}{|\pi(X^*_{a_1,a_2})|}
\lim_{\substack{t\in \BR^{\times 2}\\ t\to 0^+}}
t^{\frac{\dim \SO(2,2)-\dim T_\cpt}{2}}
\widehat{j}(X^*_{a_1,a_2},tX^\pm_{a_1,a_2})
=
-\frac{|W_{T_\mathrm{cpt}}|}{\pi(X^\pm_{a_1,a_2})}.
\]
Consequently, \eqref{eq:lemmaformula:1} equals
\[
-\left(
\frac{|\pi(X^+_{a_1,a_2})|}{\pi(X^+_{a_1,a_2})}
-
\frac{|\pi(X^-_{a_1,a_2})|}{\pi(X^-_{a_1,a_2})}
\right)
\frac{|W_{T_{\mathrm{cpt}}}|}{\pi(X^\pm_{a_1,a_2})}.
\]
A direct calculation gives
\[
\frac{\pi(X^+_{a_1,a_2})}{|\pi(X^+_{a_1,a_2})|}
=
-
\frac{\pi(X^-_{a_1,a_2})}{|\pi(X^-_{a_1,a_2})|}.
\]
Moreover, 
the number $\pi(X^\pm_{a_1,a_2})$ is real, and
\[
\pi(X^\pm_{a_1,a_2})
=
\frac{\pi(X^\pm_{a_1,a_2})}{|\pi(X^\pm_{a_1,a_2})|}
D^{\SO(2,2)}(X^\pm_{a_1,a_2})^{1/2}.
\]
Thus \eqref{eq:lemmaformula:1} equals
\[
2|W_{T_{\mathrm{cpt}}}|(\mp)\,D^{\SO(2,2)}(X^+_{a_1,a_2})^{-1/2}.
\]
It follows that
\[
\widehat{j}(\CO_+,X^\pm_{a_1,a_2})
=
\mp
|W_{T_{\mathrm{cpt}}}|
D^{\SO(2,2)}(X^+_{a_1,a_2})^{-1/2}.
\]
Together with \cite[(3.4.6)]{beuzart2015local}, this proves the lemma.
\end{proof}

We now express the germ $c_\Theta(x)$ directly as a limit of the
quasi-character $\Theta$.

\begin{cor}\label{cor:germfunaslimitofquasichar}
Let $(G,H,\xi)$ be the Gross-Prasad triple attached to an admissible pair
$(W,V)$ over $\BR$, and let $\Theta$ be a quasi-character on $G(\BR)$. Let
$x\in \Gam(G,H)$. Fix a Borel pair $(B_x,T_{\qd,x})$ for $G_x$ over $\BR$, let
\[
\Ft_{\qd,x}=\Lie T_{\qd,x},
\]
and choose a regular semisimple element
\[
X_{\qd,x}\in \Ft_{\qd,x}(\BR).
\]
Set
\[
D^{G_x}(tX_{\qd,x})
=
|\det \ad(tX_{\qd,x})|_{\Fg_x/\Ft_{\qd,x}}|,
\qquad
W_{T_{\qd,x}}=W(G_x,T_{\qd,x}).
\]
Then the following formulas hold.

\begin{enumerate}
\item If $\nil_\reg(\Fg^\p_x)$ has one element, then
\[
c_\Theta(x)
=
|W_{T_{\qd,x}}|^{-1}
\lim_{\substack{t\in \BR^{\times 2}\\ t\to 0^+}}
D^{G_x}(tX_{\qd,x})^{1/2}
\Theta(x\exp(tX_{\qd,x})).
\]

\item Suppose $\nil_\reg(\Fg^\p_x)$ has two elements. Write
\[
V^\p_x\simeq W^\p_x\oplus^\perp D^\p_x\oplus^\perp Z^\p_x
\]
and put
\[
\eta_x=
\begin{cases}
\mathrm{sig}(D^\p_x), & \dim V^\p_x \text{ even},\\
-\mathrm{sig}(D^\p_x), & \dim W^\p_x \text{ even}.
\end{cases}
\]
Then there exist regular semisimple elements
\[
X_x^+,X_x^-\in \Fg_x(\BR)
\]
which are stably conjugate but not conjugate, with common centralizer
$T_{X_x}\subset G_x$, such that
\begin{align*}
c_\Theta(x)
=&\
\frac{1}{|W_{T_{\qd,x}}|}
\lim_{\substack{t\in \BR^{\times 2}\\ t\to 0^+}}
D^{G_x}(tX_{\qd,x})^{1/2}
\Theta(x\exp(tX_{\qd,x}))
\\
&-
\frac{\eta_x}{4|W_{T_{X_x}}|}
\lim_{\substack{t\in \BR^{\times 2}\\ t\to 0^+}}
D^{G_x}(tX_x)^{1/2}
\left\{
\Theta(x\exp(tX_x^+))
-
\Theta(x\exp(tX_x^-))
\right\}.
\end{align*}
Here $X_x$ denotes either $X_x^+$ or $X_x^-$; the two choices give the same
discriminant.
\end{enumerate}
\end{cor}

\begin{proof}
For $X\in \Fg_x(\BR)$ regular semisimple, one has
\begin{equation}\label{eq:calculus}
\lim_{t\to 0}
\frac{D^G(x\exp(tX))}{D^{G_x}(tX)}
=
D^G(x).
\end{equation}
Using \cite[Prop.~4.4.1~(vi)]{beuzart2015local},
\cite[(1.8.5)]{beuzart2015local}, and \eqref{eq:calculus}, we obtain
\begin{equation}\label{eq:raphaellimitformula}
\lim_{\substack{t\in \BR^{\times 2}\\ t\to 0^+}}
D^{G_x}(tY)^{1/2}\Theta(x\exp(tY))
=
D^{G_x}(Y)^{1/2}
\sum_{\CO\in \nil_\reg(\Fg_x)}
c_{\Theta,\CO}(x)\widehat{j}(\CO,Y)
\end{equation}
for $Y\in \Fg_x(\BR)$ regular semisimple.

If $Y=X_{\qd,x}$ is split, then \cite[(3.4.7)]{beuzart2015local} gives
\begin{equation}\label{eq:limitgermforaddition:1}
\lim_{\substack{t\in \BR^{\times 2}\\ t\to 0^+}}
D^{G_x}(tX_{\qd,x})^{1/2}\Theta(x\exp(tX_{\qd,x}))
=
\frac{|W_{T_{\qd,x}}|}{|\nil_\reg(\Fg_x)|}
\sum_{\CO\in \nil_\reg(\Fg_x)}
c_{\Theta,\CO}(x).
\end{equation}
Part (1) follows immediately.

Assume now that $\nil_\reg(\Fg_x)$ has two elements. We treat the case where
$\dim V^\p_x$ is even; the other case is identical. Let
\[
\nil_\reg(\Fs\Fo(V^\p_x))
=
\{\CO_{+,V^\p_x},\CO_{-,V^\p_x}\}
\]
and let $\CO_{W^\p_x}$ be the unique regular nilpotent orbit in
$\Fs\Fo(W^\p_x)$. Then
\[
\nil_\reg(\Fg_x)
=
\{
\CO_{W^\p_x}\times \CO_{+,V^\p_x},
\CO_{W^\p_x}\times \CO_{-,V^\p_x}
\}.
\]
By Lemma \ref{lem:thecalculationofj}, one may choose
$X^+_{V^\p_x},X^-_{V^\p_x}\in \Fs\Fo(V^\p_x)$, stably conjugate but not
conjugate, so that the difference of the two limits in
\eqref{eq:raphaellimitformula} extracts the difference of the two germs.
Choose also a regular semisimple element
\[
X_{\qd,W^\p_x}\in \Ft_{\qd,W^\p_x}(\BR)
\]
for a split Cartan of $\SO(W^\p_x)$. Put
\[
X_x^\pm
=
(X_{\qd,W^\p_x},X^\pm_{V^\p_x},X_{\Fg_x^{\p\p}})
\in \Fg_x,
\]
where $X_{\Fg_x^{\p\p}}$ is a fixed regular semisimple element in
$\Fg_x^{\p\p}$. Let $T_{X_x}$ be the common centralizer of $X_x^\pm$ in $G_x$.

Applying \eqref{eq:raphaellimitformula} to $X_x^+$ and $X_x^-$, and using
Lemma \ref{lem:thecalculationofj}, gives
\begin{align}\label{eq:limitgermforsubtraction:1}
&-\frac{1}{2|W_{T_{X_x}}|}
\lim_{\substack{t\in \BR^{\times 2}\\ t\to 0^+}}
D^{G_x}(tX_x)^{1/2}
\left\{
\Theta(x\exp(tX_x^+))
-
\Theta(x\exp(tX_x^-))
\right\}
\nonumber\\
&\qquad =
c_{\Theta,\CO_{W^\p_x}\times \CO_{+,V^\p_x}}(x)
-
c_{\Theta,\CO_{W^\p_x}\times \CO_{-,V^\p_x}}(x),
\end{align}
after labeling $X_x^\pm$ consistently with the chosen parametrization of
$\CO_\pm$. Combining \eqref{eq:limitgermforsubtraction:1} with
\eqref{eq:limitgermforaddition:1} gives the stated formula. The case where
$\dim W^\p_x$ is even is obtained by replacing the indexing sign by
$-\mathrm{sig}(D^\p_x)$, as in the definition of $c_\Theta(x)$.
\end{proof}

\subsection{Stable variant}\label{subsec:stablevariant}

We recall the stable geometric multiplicity of \cite[\S 3.2]{MR3155345}, in
the notation of Subsection \ref{subsec:theformula}.

By definition,
\[
x\in \Gam(G,H)
\]
if and only if $\SO(W^{\p\p}_x)_x$ is an anisotropic torus and $G_x$ is
quasi-split. Since $W^{\p\p}_x=\Im(1-x|_W)$ has even dimension, Theorem
\ref{thm:parconjso}(1) associates to $x$, up to a subset of measure zero, a pair
\[
(\kappa^{\p\p},c^{\p\p})\in \Xi_{\reg,W^{\p\p}_x}.
\]
Moreover, by the same argument as in \cite[Lem.~5.1.0.4]{thesis_zhilin},
\[
I_{\kappa^{\p\p}}=I^*_{\kappa^{\p\p}}
\]
if and only if $\SO(W^{\p\p}_x)_x$ is anisotropic.

Let $\CC(V,W)$ be the set of pairs $(\kappa^{\p\p},c^{\p\p})$ such that
\[
(W_{\kappa^{\p\p},c^{\p\p}},q_{\kappa^{\p\p},c^{\p\p}})
\hookrightarrow (W,q_W)
\]
as an orthogonal subspace, the orthogonal complement is quasi-split, and
\[
I_{\kappa^{\p\p}}=I^*_{\kappa^{\p\p}}.
\]
Set
\[
\Xi(d_V,d_W)
=
\left\{
\kappa^{\p\p}\in \Xi_\reg
\ \middle|\
I_{\kappa^{\p\p}}=I^*_{\kappa^{\p\p}},
\quad
2|I^*_{\kappa^{\p\p}}|\leq \min\{d_W,d_V\}
\right\},
\]
where
\[
d_V=\dim V,\qquad d_W=\dim W.
\]
Theorem \ref{thm:parconjso}(1) gives, for each
$(\kappa^{\p\p},c^{\p\p})\in \CC(V,W)$, two elements
\[
x^+_{\kappa^{\p\p},c^{\p\p}},
\qquad
x^-_{\kappa^{\p\p},c^{\p\p}}
\]
in $\Gam(G,H)$. Up to measure zero, the resulting map
\[
\Gam(G,H)\longrightarrow \CC(V,W)
\]
is two-to-one. The natural projection
\[
\CC(V,W)\longrightarrow \Xi(d_V,d_W)
\]
will be used below; all measures are understood as push-forward measures.

The following lemma describes its fibers.

\begin{lem}\label{lem:fiberoverkappapp}
For $\kappa^{\p\p}\in \Xi(d_V,d_W)$, the fiber of
\[
\CC(V,W)\longrightarrow \Xi(d_V,d_W)
\]
over $\kappa^{\p\p}$ is
\[
\begin{cases}
C(\kappa^{\p\p})_{\frac{\Del_W-\Fi_{W,\kappa^{\p\p}}}{2}},
& \dim W \text{ odd},\\[4pt]
C(\kappa^{\p\p})_{\frac{\Del_V-\Fi_{V,\kappa^{\p\p}}}{2}},
& \dim W \text{ even}.
\end{cases}
\]
\end{lem}

\begin{proof}
By definition, $(\kappa^{\p\p},c^{\p\p})\in \CC(V,W)$ if and only if
\[
(W_{\kappa^{\p\p},c^{\p\p}},q_{\kappa^{\p\p},c^{\p\p}})
\hookrightarrow (W,q_W)
\]
and the orthogonal complement is quasi-split. By \eqref{eq:indexofWkappac},
this is equivalent to
\[
2\sum_{i\in I^*_{\kappa^{\p\p}}}c_i^{\p\p}
=
\Del_{W_{\kappa^{\p\p},c^{\p\p}}}
=
\Del_W-\Del_{W^\perp_{\kappa^{\p\p},c^{\p\p}}}.
\]

Assume first that $\dim W$ is odd. Then the complement
$W^\perp_{\kappa^{\p\p},c^{\p\p}}$ is odd-dimensional and quasi-split, hence
\[
\Del_{W^\perp_{\kappa^{\p\p},c^{\p\p}}}
=
\Fi_{W,\kappa^{\p\p}}.
\]
Therefore
\[
\sum_{i\in I^*_{\kappa^{\p\p}}}c_i^{\p\p}
=
\frac{\Del_W-\Fi_{W,\kappa^{\p\p}}}{2}.
\]
Conversely, this equality forces the complement to have the required
quasi-split signature. This gives the first case.

If $\dim W$ is even, the same argument is applied after embedding
$W_{\kappa^{\p\p},c^{\p\p}}$ into $V$; the relevant complement is then
odd-dimensional. One obtains
\[
\sum_{i\in I^*_{\kappa^{\p\p}}}c_i^{\p\p}
=
\frac{\Del_V-\Fi_{V,\kappa^{\p\p}}}{2}.
\]
This proves the second case.
\end{proof}

Combining Lemma \ref{lem:fiberoverkappapp} with Proposition
\ref{pro:unionpureinnerconjclassof} gives the following fixed-Kottwitz-sign
version.

\begin{lem}\label{lem:fiberoverkappappunionoverpure}
Fix $e_0\in \{\pm 1\}$.

\begin{enumerate}
\item Suppose $\dim W$ is odd. Then the fiber over
$\kappa^{\p\p}\in \Xi(d_V,d_W)$ of
\[
\bigsqcup_{\substack{
\alp\in H^1(\BR,\SO(W))\\
e(\SO(W_\alp))=e_0
}}
\CC(V_\alp,W_\alp)
\longrightarrow
\Xi(d_V,d_W)
\]
is
\[
C(\kappa^{\p\p})^{e_0\eps_{W,\kappa^{\p\p}}}.
\]

\item Suppose $\dim W$ is even. Fix an anisotropic line $D$ with
\[
\mathrm{sig}(D)\in \{\pm 1\}.
\]
Then the fiber over $\kappa^{\p\p}\in \Xi(d_V,d_W)$ of
\[
\bigsqcup_{\substack{
\alp\in H^1(\BR,\SO(W\oplus^\perp D))\\
e(\SO(W_\alp\oplus^\perp D))=e_0
}}
\CC(V_\alp,W_\alp)
\longrightarrow
\Xi(d_V,d_W)
\]
is
\[
C(\kappa^{\p\p})^{e_0\eps_{V,W,\kappa^{\p\p},D}},
\]
where
\[
\eps_{V,W,\kappa^{\p\p},D}
=
(-1)^{
\frac{
|I^*_{\kappa^{\p\p}}|
-
\frac{\Del_V-\Fi_{V,\kappa^{\p\p}}}{2}
+
\frac{\dim W+1+\Del_W+\mathrm{sig}(D)}{2}
-
\PI(W,D)
}{2}
}.
\]
\end{enumerate}
Here $\eps_{W,\kappa^{\p\p}}$ is the sign defined in Proposition
\ref{pro:unionpureinnerconjclassof}.
\end{lem}

Let $\Theta$ be a \textbf{stable} quasi-character on $G(\BR)$ in the sense of
\cite[\S 12.1]{beuzart2015local}. Then, by \cite[(10.1.2)]{thesis_zhilin} or
equivalently by Corollary \ref{cor:germfunaslimitofquasichar}, since the
elements $X_x^+$ and $X_x^-$ are stably conjugate, one has
\[
c_\Theta(x)
=
|\nil_\reg(\Fg_x)|^{-1}
\sum_{\CO\in \nil_\reg(\Fg_x)}
c_{\Theta,\CO}(x),
\qquad x\in \Gam(G,H).
\]
Following \cite[\S 3.2]{MR3155345}, define
\begin{equation}\label{eq:stablegeomultiplicity}
m_\geom^S(\Theta)
=
\int_{\kappa^{\p\p}\in \Xi(d_V,d_W)}
2^{|I^*_{\kappa^{\p\p}}|}
D^G(\kappa^{\p\p})^{1/2}
c_\Theta(\kappa^{\p\p})
\Del(\kappa^{\p\p})^{-1/2}\,d\kappa^{\p\p}.
\end{equation}
Here
\[
D^G(\kappa^{\p\p})
=
D^G(x^\pm_{\kappa^{\p\p},c^{\p\p}}),
\qquad
\Del(\kappa^{\p\p})
=
\Del(x^\pm_{\kappa^{\p\p},c^{\p\p}}),
\]
and
\[
c_\Theta(\kappa^{\p\p})
=
\begin{cases}
c_\Theta(x^+_{\kappa^{\p\p},c^{\p\p}})
=
c_\Theta(x^-_{\kappa^{\p\p},c^{\p\p}}),
&
(\kappa^{\p\p},c^{\p\p})\in \CC(V,W),\
2|I^*_{\kappa^{\p\p}}|<\dim W,
\\[6pt]
\Theta(x^+_{\kappa^{\p\p},c^{\p\p}})
+
\Theta(x^-_{\kappa^{\p\p},c^{\p\p}}),
&
(\kappa^{\p\p},c^{\p\p})\in \CC(V,W),\
2|I^*_{\kappa^{\p\p}}|=\dim W.
\end{cases}
\]

\section{The reduction and the basic cases}\label{sec:redtoendoscopypf}

In this section we establish Theorem \ref{thm:main}, and hence complete the
proof of Conjecture \ref{conjec:gp}. In Subsection \ref{subsec:steps}, we
describe the reduction procedure and prove the required parabolic reductions.
In Subsection \ref{subsec:theproof}, we complete the dimension
induction using the classical toric restriction case. In Subsection \ref{subsec:endoscopic}, we establish the
endoscopic reduction, following the strategy of \cite[Prop.~3.3]{MR3155345}
and using the results of Sections \ref{sec:regconj} and
\ref{sec:geomultiplicity}.

\subsection{Reduction strategy}\label{subsec:steps}

We now describe the reduction steps.

\begin{num}
\item\label{step:1} We first reduce to the case where the tempered local
$L$-parameters are of $\RO$-type in the sense of Subsection
\ref{subsec:conjecofggp}, using \cite[Thm.~4.4]{moeglin2020paquets} and
\cite[\S 6]{renard2024generic}. Following
\cite[\S 4.1]{moeglin2020paquets}, this is equivalent to requiring the
parameters to have \textbf{good parity}.

\item\label{step:2} We make further parabolic reductions for which the
natural map of component groups is an isomorphism.

\item\label{step:3} We apply proper elliptic endoscopic reduction to
component-group generators and conclude by induction, with the classical
$\SO(2)\times\SO(3)$ case as the nontrivial base case.
\end{num}

Let $V$ be a non-degenerate quadratic space over $\BR$, and let $\vphi_V$ be a
tempered local $L$-parameter for $\SO(V)$. As recalled in Subsection
\ref{subsec:conjecofggp}, the representation $\std\circ \vphi_V$ admits a
natural decomposition
\[
\RM_V=\bigoplus_i m_i\RM_i.
\]
The irreducible representations $\RM_i$ are classified into $\RO$-type,
$\Sp$-type, and $\GL$-type, with corresponding index sets $\RI_{\RO}$,
$\RI_\Sp$, and $\RI_{\GL}$, respectively. Moreover, for $i\in \RI_{\Sp}$, the
multiplicity $m_i$ is even. For each unordered pair indexed by
$i\in\RI_\GL$, choose one representative $\RM_i$.

Following the notation of \cite[\S 4.1]{moeglin2020paquets} and
\cite[\S 6]{renard2024generic}, the $\RO$-type part is the
\textbf{good-parity} part of the parameter, while the $\Sp$-type and
$\GL$-type parts form the \textbf{bad-parity} part. In
\cite[\S~4.1]{moeglin2020paquets}, the corresponding French terms are
``bonne parité'' and ``mauvaise parité''. Thus the above decomposition may be
rewritten as
\[
\RM_V=\RM_{gp}\oplus^{\perp} (\RM_{bp}\oplus (\RM_{bp})^{\vee}),
\]
where
\[
\RM_{gp}:=\bigoplus_{i\in\RI_{\RO}}m_i\RM_i
\quad\text{and}\quad
\RM_{bp}:=
\bigg(\bigoplus_{i\in \RI_{\Sp}}\frac{m_i}{2}\RM_i\bigg)
\oplus
\bigoplus_{i\in\RI_{\GL}}m_i\RM_i.
\]

\subsubsection{Parabolic reduction to the good parity case}

We first address \eqref{step:1}, working on the pure inner forms for
which the packet is nonempty. The decomposition of
$\RM_V = \std\circ \vphi_V$ gives an orthogonal decomposition
\[
V=V_{gp}\oplus(X_{bp}\oplus X^\vee_{bp}),
\]
where $X_{bp}$ is a totally isotropic subspace of the non-degenerate split
quadratic space $X_{bp}\oplus X^\vee_{bp}$, such that the parameter $\vphi_V$
factors through the Levi subgroup
\[
{}^L\big(\GL(X_{bp})\big)\times_{\Gal(\BC/\BR)} {}^L\big( \SO(V_{gp}) \big)\subset {}^L\big( \SO(V) \big).
\]
We choose the Levi embedding for the fixed full $L$-parameter; when
$V_{gp}=0$, this also fixes the Siegel parabolic. We denote the factorization
of the parameter through the $L$-group of the Levi by $\vphi^\GL_{bp}\boxtimes \vphi_{V_{gp}}$. By \eqref{eq:component}, the
component group depends only on the good-parity part of the local
$L$-parameter. Therefore,
\[
|\CS_{\vphi_{V_{gp}}}| = |\CS_{\vphi_V}|.
\]

By \cite[Th\'eor\`eme 4.4]{moeglin2020paquets} and
\cite[\S 6]{renard2024generic}, parabolic induction induces a bijection from
the tempered local $L$-packet $\Pi_{\vphi_{V_{gp}}}$ attached to
$\vphi_{V_{gp}}$ onto the tempered local $L$-packet $\Pi_{\vphi_{V}}$ attached
to $\vphi_{V}$. More precisely, this bijection is given by
\[
\pi\mapsto \Ind_{P_X}^{\SO(V)}(\sig\boxtimes \pi),
\]
and it induces an isomorphism
\[
\CS_{\vphi_{V_{gp}}}\simeq \CS_{\vphi_V}.
\]
Here $P_X$ is a parabolic subgroup of $\SO(V)$ with Levi subgroup $\GL(X_{bp})\times \SO(V_{gp})$, and $\sig$ is the representation of $\GL(X_{bp})(\BR)$ attached to $\vphi^\GL_{bp}$.

By Theorem \ref{thm:luothesis} and
\cite[Lem.~7.4.0.2]{thesis_zhilin}, the multiplicity is preserved under this
parabolic induction. More precisely, for $U=W,V$, let $P_U$ have Levi subgroup
$\GL(X_{U,bp})\times\SO(U_{gp})$, let $\sigma_U$ correspond to
$\vphi^\GL_{U,bp}$, and let $\pi_{U_{gp}}\in\Pi_{\vphi_{U_{gp}}}$. Then
\[
\begin{split}
&m\!\left(
\Ind_{P_W}^{\SO(W)}(\sigma_W\boxtimes\pi_{W_{gp}})
\boxtimes
\Ind_{P_V}^{\SO(V)}(\sigma_V\boxtimes\pi_{V_{gp}})
\right)\\
&\hspace{35mm}=m(\pi_{W_{gp}}\boxtimes\pi_{V_{gp}}),
\end{split}
\]
where induction is normalized and the multiplicity on the right is taken
for the residual Gross-Prasad pair, interchanging its two spaces if
necessary. Moreover, the distinguished character defined in
\eqref{eq:chivphicharacter:1} is preserved under the above identifications of
component groups.

Therefore, the proof of Conjecture
\ref{conjec:gp} reduces to the case where both $\vphi_V$ and $\vphi_W$ have
good parity in the sense of \cite[\S 4]{moeglin2020paquets}.

\subsubsection{Parabolic reduction preserving the component group}

We use the following compatibility of tempered $L$-packets with parabolic
induction. All induction in this section is normalized, and the Whittaker
normalizations on a Levi subgroup and the ambient group are compatible.

\begin{lem}\label{lem:component-preserving-induction}
Let $V=V_0\oplus^\perp(X\oplus X^{\vee})$ be the decomposition of a non-degenerate quadratic space, with $V_0\neq 0$ and $X,X^{\vee}$ totally isotropic. Let $\varphi^{\GL}$ be a tempered local $L$-parameter of $\GL(X)$ and $\Pi_{\varphi^{\GL}}=\{\sigma\}$. Let $\varphi_{V_0}$ be a tempered local $L$-parameter of $\SO(V_0)$, and let $\varphi_V$ be the tempered local $L$-parameter of $\SO(V)$ defined by the composition of $\varphi^{\GL}\boxtimes\varphi_{V_0}$ with the Levi embedding
\[
{}^L\big(\GL(X)\big)_{\times_{\Gal(\BC/\BR)}}{}^L\big(\SO(V_0)\big)\to{}^L\SO(V).
\]
Suppose that the natural map
\[
\iota:\CS_{\varphi_{V_0}}\longrightarrow\CS_{\varphi_V}
\]
is an isomorphism. Then
\begin{align*}
\Pi_{\varphi_{V_0}}^{\mathrm{Vogan}}&\to \Pi_{\varphi_{V}}^{\mathrm{Vogan}}
\\
\pi_{V_0} &\mapsto \sigma\rtimes \pi_{V_0}
\end{align*}
defines a bijection between the Vogan $L$-packets.
Every displayed induced representation is irreducible, and
\[
\chi_{\sigma\rtimes\pi_{V_0}}\circ\iota=\chi_{\pi_{V_0}}.
\]

\end{lem}

\begin{proof}
For real tempered $L$-packets, the compatibility of the internal parametrization
with parabolic induction is described by the dual $R$-group; see
\cite[\S9--\S10]{shelstad2008tempered}. The irreducible constituents of the
induction are labelled by the characters of $\CS_{\vphi_V}$ whose
restriction along $\iota$ is $\chi_{\pi_{V_0}}$. When $\iota$ is an
isomorphism there is exactly one such character and the relative $R$-group
is trivial. This gives irreducibility, the stated equality of characters,
and the bijection, with the pure inner form retained throughout.
\end{proof}

For a good-parity parameter, removing two copies of a constituent of
multiplicity at least three leaves the component group unchanged, by
\eqref{eq:SvphiVdescription}. We may therefore assume that every irreducible
constituent has multiplicity one or two. There is one further reduction in
the even orthogonal case. If a character $\eta$ is the only
one-dimensional constituent, its multiplicity is two. If any other
constituent occurs, we may remove $\eta\oplus\eta$: the determinant
condition makes its $\RO(2,\BC)$ factor a connected $\SO(2,\BC)$ factor,
so the natural map of component groups is again an isomorphism. The
remaining orthogonal space is nonzero. If no other constituent occurs,
we retain the split torus $\SO(1,1)$, whose component group is trivial.

The multiplicity formula for parabolic induction and the compatibility of
\eqref{eq:chivphicharacter:1} with the above Levi embedding show that these
reductions preserve the assertion to be proved. Here it is the entire
Gross-Prasad character, including its determinant factors, that is
preserved.

\subsubsection{Elliptic endoscopic reduction}

\begin{defin}\label{def:endoscopic}
An element $s\in S_{\vphi_V}$ satisfying $s^2=1$ defines a
\textbf{proper elliptic endoscopic datum} if its connected centralizer
\[
\widehat H=Z_{\wh{\SO(V)}}(s)^\circ
\subsetneq\wh{\SO(V)}
\]
is proper and
\[
\bigl(Z(\widehat H)^{\Gal(\BC/\BR)}\bigr)^\circ
\subset Z(\wh{\SO(V)}).
\]
The Galois action here is the one induced by the factorization of
$\vphi_V$ through the $L$-group of the endoscopic group. Its two orthogonal factors
have standard parameter spaces equal to the $\pm1$-eigenspaces of
$s$.
\end{defin}

In applying Proposition \ref{pro:endoscopicid}, we use actual involutive
lifts in the centralizers. For an even-dimensional orthogonal group the
two endoscopic quadratic spaces have dimensions equal to those of the
$\pm1$-eigenspaces; for an odd-dimensional orthogonal group these
dimensions are larger by one. Thus a zero standard parameter space gives
$\SO(0)$ or $\SO(1)$, respectively, both trivial. We also allow scalar
lifts $\pm1$, which give a full factor and a trivial factor. At least one
of the two data in every inductive application below is proper elliptic.

In particular, an anisotropic real torus $\SO(2)$ is an allowed elliptic
factor: its complex dual has a one-dimensional center, but Galois acts on
that center by inversion. Thus ellipticity does not mean that the complex
centralizer has finite center.

We now explain explicitly which component-group elements are needed.
Every irreducible representation of $W_{\BR}$ has dimension one or two,
and its only self-dual one-dimensional representations are $1$ and
$\operatorname{sgn}$. For every two-dimensional good-parity constituent,
choose the involution which acts by $-1$ on one copy of this constituent
and by $+1$ on all remaining copies and constituents. If both $1$ and
$\operatorname{sgn}$ occur, also choose the involution acting by $-1$ on
one copy of each and by $+1$ on the complement. By
\eqref{eq:SvphiVdescription}, the images of these involutions generate
$\CS_{\vphi_V}$.

For the reduced parameters above, either the whole standard parameter
space has dimension two, or each of these involutions defines a proper
elliptic endoscopic datum. Indeed, its $(-1)$-eigenspace has dimension two. In the
symplectic case it gives a split $\SO(3)$ factor. In the orthogonal case it
is either an irreducible orthogonal Weil representation of dimension two
or $1\oplus\operatorname{sgn}$; in both cases its determinant is
$\operatorname{sgn}$, so it gives an anisotropic $\SO(2)$ factor. The
complement is nonzero unless the original standard space has dimension
two. The only possible obstruction to ellipticity of a nonzero complement
would be a split $\SO(2)$ factor. For a two-dimensional irreducible
constituent this would require a complementary $\eta\oplus\eta$, already
removed above. For the generator attached to $1$ and
$\operatorname{sgn}$ it would require multiplicities $(3,1)$ or $(1,3)$,
also already removed. Thus the complementary factor is elliptic as well.

\subsection{The basic cases}\label{subsec:theproof}

We argue by induction on $\dim\SO(W)+\dim\SO(V)$. By the known Part (1) of Conjecture \ref{conjec:gp},
\[
F_\vphi(s):=
\sum_{\pi\in\Pi^\Vg_\rel(\vphi)}m(\pi)\chi_\pi(s)
\]
is a character of $\CS_\vphi$. It is therefore enough to compare
$F_\vphi$ and $\chi_\vphi$ on the generators just constructed, placed in
one factor of $\CS_{\vphi_W}\times\CS_{\vphi_V}$, with the identity in
the other factor. If the corresponding datum is proper elliptic,
Proposition \ref{pro:endoscopicid} and the computation following it give
the desired equality from the induction hypothesis. At least one
orthogonal factor has strictly decreased, so each Gross-Prasad pair
occurring on the right has smaller total group dimension.

If one factor has a standard parameter space of dimension two and the other
has a proper elliptic generator $a_0$, we also apply Proposition
\ref{pro:endoscopicid} to $(a_0,b)$, where $b$ is a generator in the small
factor. The lift of $b$ is scalar and hence gives the full group and a
trivial factor on that side; the proper decomposition on the other side
still makes both smaller pairs eligible for induction. Comparing this
identity with the one for $(a_0,1)$ and using multiplicativity determines
the value on $(1,b)$. A trivial component group requires no further check.

If no such proper elliptic datum remains in either factor, the standard
parameter spaces have dimension at most two. Hence the only nontrivial
pair left, up to permutation, has dimensions $(2,3)$. This is the toric restriction problem
for $\PGL_2$ and its quaternionic inner form, for which the epsilon
criterion is classical; see \cite[\S\S12,15]{MR1186476}. Cases with a
trivial orthogonal group are the Whittaker case (or the one-dimensional
representations of a torus), and the assertion follows from the chosen
Whittaker normalization. This completes the reduction to the endoscopic
identity proved below.

\subsection{Endoscopic reduction}\label{subsec:endoscopic}
In this subsection, we prove the endoscopic reduction introduced in Subsection
\ref{subsec:steps}, using the results of the preceding sections.

We recall the setup for the endoscopic reduction. Let
$\vphi=\vphi_W\times \vphi_V$ be a tempered local $L$-parameter of
$G=\SO(W)\times \SO(V)$. 
For every element $s_V\in \mathcal{S}_{\varphi_V}$, from \ref{eq:SvphiVdescription}, there is a lift $\tilde s_V\in S_{\varphi_V}$ of $s_V$ such that $\wt s_V^2=1$. Then $\tilde s_V$ determines an endoscopic group $\SO(V^+)\times \SO(V^-)$ of $\SO(V)$ as in Definition \ref{def:endoscopic}. The standard parameter spaces of the two factors are
\[
  \RM_{V_+} = \ker(\tilde{s}_V - 1), \qquad
  \RM_{V_-} = \ker(\tilde{s}_V + 1).
\]
We use the notation $\RM_V^{s_V=\pm 1}$ for simplification, and we define $\RM_W^{s_W=\pm 1}$ for $s_W\in \CS_{\varphi_W}$ similarly.


In the situation when there is a proper elliptic datum, $S_{Z_V}:=Z_{\GL(V_{\BC})}\cap S_{\varphi_V}$ is a proper subgroup of $S_{\varphi_V}$. Since both $\chi_\vphi$ and $\chi_{\pi_\vphi}$ are characters of $\CS_\vphi$, it is enough to prove
\[
\chi_{\pi_\vphi} = \chi_\vphi
\quad \text{on}\quad
\big(\CS_{\vphi_V}\bs \CS_{Z_V} \big)\times \CS_{\vphi_W}.
\]
We exclude $\CS_{Z_V}$ because the endoscopic reduction applies precisely to the part where it reduces the problem to smaller-rank Gross-Prasad triples.

There is a unique representation $\pi_\vphi\in \Pi_\rel^\Vg(\vphi)$ with $m(\pi_\vphi)=1$. Therefore, for any $s\in \CS_\vphi$,
\[
\chi_{\pi_\vphi}(s) = \sum_{\pi\in \Pi^\Vg_\rel(\vphi)}\chi_{\pi}(s) m(\pi).
\]
Thus, in the endoscopic case, it remains to prove the following identity for
any $s\in \CS_\vphi$:
\begin{equation}\label{eq:endoscopic:1}
\sum_{\pi\in \Pi^\Vg_\rel(\vphi)}\chi_{\pi}(s)
m(\pi) = \chi_\vphi(s).
\end{equation}

We introduce the following notation.

\begin{defin}\label{defin:virtualrep}
Let $V$ be a non-degenerate quadratic space over $\BR$, and let $\vphi_V$ be a
local $L$-parameter of $\SO(V)$. For $s\in \CS_{\vphi_V}$, define the virtual
representation
\[
\Sig^s_{V,\vphi_V} =
\sum_{\pi\in \Pi^{\SO(V)}(\vphi_V)}
\chi_\pi(s) \pi.
\]
When $\vphi_V$ is clear, we abbreviate it as $\Sig^s_{V}$.
\end{defin}

\begin{defin}\label{defin:stablemultiplicity}
Let $(W,V)$ be an admissible pair over $\BR$ with Gross--Prasad triple
$(G,H,\xi)$. Let $\vphi = \vphi_W\times \vphi_V$ be a local $L$-parameter of
$G=\SO(W)\times \SO(V)$. Define
\[
m^S_{W,V,\vphi} :=
\sum_{\alp\in \RH^1(\BR,\SO(W))}m(
\Sig^{-1_{\vphi_W}}_{W_\alp},
\Sig^{1_{\vphi_V}}_{V_\alp}
).
\]
Here
\[
m(\Sig^{-1_{\vphi_W}}_{W},\Sig^{1_{\vphi_V}}_{V}) :=
\sum_{
\pi\in \Pi^{G}(\vphi)}
\chi_\pi(-1_{\vphi_W},1_{\vphi_V})
 m(\pi)
\]
and $\pm 1_{\vphi_*}$ is the $\pm$ identity element on the component group.
When $\vphi_*$ is clear, we abbreviate $1_{\vphi_*}$ as $1$.
\end{defin}

For $s_V\in \CS_{\vphi_V}\bs \CS_{Z_V}$, the eigenspaces
\[
\RM_V^{s_V = \pm 1}
\quad \text{and}\quad 
\RM_W^{s_W = \pm 1}
\]
are both of even dimension and strictly smaller than $\dim V$ and $\dim W$,
respectively. Following \cite[\S 1.7~\&~\S 3.3]{MR3155345}, we have the
following lemma.

\begin{lem}
There exist elliptic endoscopic groups $\SO(V_+)\times \SO(V_-)$ of
$\SO(V)$ and $\SO(W_+)\times \SO(W_-)$ of $\SO(W)$ such that the standard
representation spaces of ${}^L\big(\SO(V_\pm)\big)$ and
${}^L\big(\SO(W_\pm)\big)$ are $\RM_V^{s_V=\pm1}$ and
$\RM_W^{s_W=\pm1}$, respectively, and such that, up to permutation, every
pair $(W_\epsilon,V_\delta)$, $\epsilon,\delta=\pm$, is admissible over
$\BR$.
\end{lem}

\begin{proof}
Using the notation in \eqref{conjec:pureinner:1}, set
$
\Del(V)=\mathrm{PI}(V)-\mathrm{NI}(V).
$
We may assume that $\dim V$ is odd, and hence that $\dim W$ is even. We may
also replace $(W,V)$ by the relevant quasi-split pure inner form. Then
\[
\Del(V)=\pm1,\qquad \Del(W)\in\{0,\pm2\},
\]
and, since $(W,V)$ is admissible up to permutation, the admissibility condition
is equivalent to
\[
\eta:=\Del(V)-\Del(W)\in\{\pm1\}.
\]
Hence
$
\Del(W)\in\{0,-2\eta\}.
$
By \cite[\S 1.7]{MR3155345}, the elliptic endoscopic groups of $\SO(W)$ are
parametrized by pairs of even-dimensional quasi-split quadratic spaces
$W_\pm$ satisfying
\begin{align*}
\dim W_+ +\dim W_-&=\dim W
\\
\disc(W_+)\disc(W_-)=\disc(W)
&\Longleftrightarrow
\Del(W_+)+\Del(W_-)\equiv \Del(W)\pmod{4}.
\end{align*}
We choose the representatives $W_\pm$ so that
\[
\dim W_\pm=\dim\RM_W^{s_W=\pm1},\qquad
\Del(W_\pm)\in\{0,-2\eta\}.
\]
Similarly, the elliptic endoscopic groups of $\SO(V)$ are
$\SO(V_+)\times\SO(V_-)$ with
\[
\dim V_+ +\dim V_-=1+\dim V.
\]
We choose the split representatives $V_\pm$ so that
\[
\dim V_\pm=1+\dim\RM_V^{s_V=\pm1},\qquad
\Del(V_\pm)=-\eta.
\]
Then, for every $\epsilon,\delta=\pm$,
\[
\Del(V_\delta)-\Del(W_\epsilon)=\pm\eta=\pm1.
\]
Thus, up to permutation, each pair $(W_\epsilon,V_\delta)$ is admissible over
$\BR$.
\end{proof}

We complete the proof of Theorem \ref{thm:main} with the following proposition.
\begin{pro}\label{pro:endoscopicid}
Fix an admissible pair $(W,V)$ over $\BR$ with Gross-Prasad triple $(G,H,\xi)$. Let $\vphi = \vphi_W\times \vphi_V$ be a tempered local $L$-parameter of $G$, and let $s=(s_V,s_W)\in \CS_\vphi = \CS_{\vphi_V}\times \CS_{\vphi_W}$. For $e_0\in \{\pm 1\}$, the following identity holds:
\begin{align}\label{eq:endoscopicid}
\sum_{
\substack{
\alp\in \RH^1(\BR,\SO(W))
\\
e(G_\alp) = e_0
}
}
m(\Sig^{s_W}_{W_\alp},\Sig^{s_V}_{V_\alp})
=
\frac{1}{2}
\big(
e_0 m^S_{W_+,V_+} m^S_{W_-,V_-}
+
m^S_{W_+,V_-} m^S_{W_-,V_+}
\big).
\end{align}
Here
\[
m(\Sig^{s_W}_{W_\alp},\Sig^{s_V}_{V_\alp}) =
\sum_{
\pi\in \Pi^{G_\alp}(\vphi)}
\chi_\pi(s)
 m(\pi)
,\qquad
m^S_{W_\pm,V_\pm} = m^S_{W_\pm,V_\pm,\vphi_{W_\pm}\oplus \vphi_{V_\pm}}.
\]
\end{pro}

We postpone the proof of Proposition \ref{pro:endoscopicid} to the final subsubsection. We now use Proposition \ref{pro:endoscopicid} to prove \eqref{eq:endoscopic:1}.

Adding \eqref{eq:endoscopicid} for $e_0=\pm 1$,
\[
\sum_{\alp\in \RH^1(\BR,\SO(W))}
m(\Sig^{s_W}_{W_\alp},\Sig^{s_V}_{V_\alp}) =
m^S_{W_+,V_-} m^S_{W_-,V_+}.
\]
The left-hand side is
\[
\sum_{\alp\in \RH^1(\BR,\SO(W))}
\sum_{\pi\in \Pi^{G_\alp}(\vphi)}
\chi_\pi(s) m(\pi)
=
\sum_{\pi\in \Pi^\Vg_\rel(\vphi)}
\chi_\pi(s) m(\pi).
\]
Therefore, to establish \eqref{eq:endoscopic:1}, it suffices to show that
\[
m^S_{W_+,V_-} m^S_{W_-,V_+} =
\chi_\vphi(s).
\]
By Definition \ref{defin:stablemultiplicity},
\[
m^S_{W_+,V_-}
=
\sum_{\pi\in \Pi^\Vg_{\rel}(\vphi_{W_+}\times \vphi_{V_-})}
\chi_\pi(-1,1) m(\pi).
\]
By the induction hypothesis applied to the smaller Gross-Prasad pairs, we have
\[
m^S_{W_+,V_-} =
\chi_{\vphi_{W_+}\times \vphi_{V_-}}(-1,1),
\quad
m^S_{W_-,V_+} =
\chi_{\vphi_{W_-}\times \vphi_{V_+}}(-1,1).
\]
Hence it remains to prove
\[
\chi_{\vphi_{W_+}\times \vphi_{V_-}}(-1,1)
\chi_{\vphi_{W_-}\times \vphi_{V_+}}(-1,1)
=
\chi_\vphi(s).
\]
By \cite[(10.4)]{MR1186476} and \cite[Prop.~10.5]{MR1186476}, it suffices to show
\[
\chi_{\vphi_{W_+}\times \vphi_{V_-}}(1,-1)
\chi_{\vphi_{W_-}\times \vphi_{V_+}}(1,-1)
=
\chi_\vphi(s).
\]

By \eqref{eq:chivphicharacter:1},
\begin{align*}
\chi_{\vphi_{W_+}\times \vphi_{V_-}}(1,-1)
= &
\det
({\RM^{s_V=-1}_V})^{\frac{\dim \RM_{W}^{s_W = 1}}{2}}(-1)
\det
({\RM_{W}^{s_W=1}})^{\frac{\dim \RM_V^{s_V=-1}}{2}}(-1)
\\
&
\veps
\big(
\frac{1}{2},
\RM^{s_V=-1}_V\otimes \RM_{W}^{s_W = 1},\psi
\big),
\end{align*}
and
\begin{align*}
\chi_{\vphi_{W_-}\times \vphi_{V_+}}(1,-1)
=
&
\det
({\RM^{s_W=-1}_W})^{\frac{\dim \RM_{V}^{s_V = 1}}{2}}(-1)
\det
({\RM_{V}^{s_V=1}})^{\frac{\dim \RM_W^{s_W=-1}}{2}}(-1)
\\
&
\veps
\big(
\frac{1}{2},
\RM^{s_W=-1}_W\otimes \RM_{V}^{s_V = 1},\psi
\big).
\end{align*}
Comparing this with the definition of $\chi_\vphi$, it remains only to show
\begin{align*}
&\veps
\big(
\frac{1}{2},
\RM^{s_V=-1}_V
\otimes
\RM^{s_W=-1}_W,
\psi
\big)^2
\det({\RM^{s_V=-1}_V})^{\frac{\dim \RM_W^{s_W=-1}}{2}}(-1)
\det({\RM_W^{s_W=-1}})^{\frac{\dim \RM_V^{s_V=-1}}{2}}(-1)
\\
&
\det({\RM_W^{s_W=-1}})^{
\frac{\dim \RM^{s_V=-1}_V}{2}
}(-1)
\det({\RM_V^{s_V=-1}})^{\frac{\dim \RM_W^{s_W=-1}}{2}}(-1)=1.
\end{align*}
This follows because $\veps(\cdots)\in \{\pm 1\}$. This completes the proof of
Theorem \ref{thm:main} in the endoscopic case.

\subsubsection{Proof of Proposition \ref{pro:endoscopicid}}

Let
\[
\Theta_{\Sig^{s_V}_V}
=
\sum_{\pi\in \Pi^{\SO(V)}(\vphi_V)}\chi_\pi(s_V)\Theta_\pi
\]
be the virtual distribution character attached to $\Sig^{s_V}_V$. We use the
following consequences of Shelstad's work.
\begin{num}
\item\label{num:shelstadworkendoscopy}
\begin{enumerate}
    \item
By \cite[Lem.~5.2]{shelsteadinner}, the distribution characters $\Theta_{\Sig^{1}_{V_{\pm}}}$ and $\Theta_{\Sig^{1}_{W_\pm}}$ are stable.

\item By \cite{MR532374,shelsteadinner,MR627641,MR2454336,MR2581952,MR2448289},
\[
e(\SO(V_\alp)) \Theta_{\Sig^{s}_{V_\alp}}
\quad \text{respectively}\quad
e(\SO(W_\alp))\Theta_{\Sig^s_{W_\alp}}
\]
is the endoscopic transfer (see, for instance, \cite[\S 1.6]{MR3155345}) of
$\Theta_{\Sig^1_{V_+}}\times \Theta_{\Sig^1_{V_-}}$, respectively of
$\Theta_{\Sig^1_{W_+}}\times \Theta_{\Sig^1_{W_-}}$, for any
$\alp\in H^1(\BR,\SO(V))$, respectively
$\alp\in H^1(\BR,\SO(W))$.
\end{enumerate}
\end{num}

The following proposition is the analogue of \cite[\S 3.3, Prop.~3.1]{MR3155345}.

\begin{pro}\label{pro:analoguewaldspurger3.3}
For $e_0 = \pm 1$,
\begin{align*}
\sum_{
\substack{
\alp\in H^1(\BR,\SO(W))
\\
e(G_\alp) = e_0
}
}
m_\geom(\Sig^{s_W}_{W_\alp},
\Sig^{s_V}_{V_\alp}
) = &
\frac{1}{2}
\big(
e_0 m^S_\geom(\Sig^1_{W_+},\Sig^1_{V_+}) m^S_\geom(\Sig^1_{W_-},\Sig^1_{V_-})
\\
&+
m^S_\geom(\Sig^1_{W_+},\Sig^1_{V_-})
m^S_\geom(\Sig^1_{W_-},\Sig^1_{V_+})
\big).
\end{align*}
Here
\[
m^S_\geom(\Sig^{s_W}_W,\Sig^{s_V}_V) =
m^S_\geom(\Theta_{\Sig^{s_W}_W}\times \Theta_{\Sig^{s_V}_V}),
\]
where the latter is defined in \eqref{eq:stablegeomultiplicity}.
\end{pro}
\begin{proof}
Using \eqref{num:shelstadworkendoscopy}, Lemmas \ref{lem:fiberoverkappapp}
and \ref{lem:fiberoverkappappunionoverpure}, and Corollary
\ref{cor:germfunaslimitofquasichar}, the proof of
\cite[Prop.~3.3]{MR3155345} applies verbatim. For convenience, assume that
$\dim W$ is odd; the other case is similar. By \eqref{eq:geomultiplicity} and
Lemma \ref{lem:fiberoverkappappunionoverpure},
\[
\sum_{
\substack{
\alp\in H^1(\BR,\SO(W))
\\
e(G_\alp) = e_0
}
}
m_\geom(\Sig^{s_W}_{W_\alp},
\Sig^{s_V}_{V_\alp}
)
=
\int_{\Xi(d_V,d_W)}
f_{1,e_0}(\kappa^{\p\p})\ud \kappa^{\p\p}
\]
where
\[
f_{1,e_0}(\kappa^{\p\p}) =
\sum_{
    c^{\p\p}\in C(\kappa^{\p\p})^{e_0\cdot
    \eps_{W,\kappa^{\p\p}}}
}
c_{\Sig^{s_W}_W,\Sig^{s_V}_V}(x_{\kappa^{\p\p},c^{\p\p}})
D^G(x_{\kappa^{\p\p},c^{\p\p}})^{1/2}\Del(x_{\kappa^{\p\p},c^{\p\p}})^{-1/2}
\]
and
$
c_{\Sig^{s_W}_W,\Sig^{s_V}_V}=
c_{\Theta_{\Sig^{s_W}_W}\times\Theta_{\Sig^{s_V}_V}}.
$
By Remark \ref{rmk:afterpro:unionpureinnerconjclassof},
$\eps_{W,\kappa^{\p\p}}$ depends only on $\kappa^{\p\p}$ and $e_0$. Similarly,
let $G^{\pm\pm} = \SO(W_{\pm})\times \SO(V_\pm)$. By
\eqref{eq:stablegeomultiplicity} and a direct computation,
\begin{align*}
&\frac{1}{2}
\big(
e_0 m^S_\geom(\Sig^1_{W_+},\Sig^1_{V_+})
m^S_\geom(\Sig^1_{W_-},\Sig^1_{V_-})
+
m^S_\geom(\Sig^1_{W_+},\Sig^1_{V_-})
m^S_\geom(\Sig^1_{W_-},\Sig^1_{V_+})
\big)
\\
=&
\int_{\Xi(d_V,d_W)}
f_{2,e_0}(\kappa^{\p\p})\ud \kappa^{\p\p}
\end{align*}
where $f_{2,e_0}(\kappa^{\p\p})$ is equal to
\begin{align*}
2^{|I^*_{\kappa^{\p\p}}|-1}
e_0
\sum_{
(I^{\p\p}_1,I^{\p\p}_2)\in \CI^+(\kappa^{\p\p})
}
&\big\{c_{\Sig_{W_+}^1,\Sig^1_{V_+}}
(\kappa^{\p\p}(I^{\p\p}_1))
D^{G_{++}}(\kappa^{\p\p}(I_1^{\p\p}))^{1/2}
\Del(\kappa^{\p\p}(I_1^{\p\p}))^{1/2}
\\
&c_{\Sig_{W_-}^1,\Sig^1_{V_-}}
(\kappa^{\p\p}(I^{\p\p}_2))
D^{G_{--}}(\kappa^{\p\p}(I_2^{\p\p}))^{1/2}
\Del(\kappa^{\p\p}(I_2^{\p\p}))^{1/2}
\big\}
\end{align*}
\begin{align*}
+2^{|I^*_{\kappa^{\p\p}}|-1}
\sum_{
(I^{\p\p}_1,I^{\p\p}_2)\in \CI^-(\kappa^{\p\p})
}
&\big\{c_{\Sig_{W_+}^1,\Sig^1_{V_-}}
(\kappa^{\p\p}(I^{\p\p}_1))
D^{G_{+-}}(\kappa^{\p\p}(I_1^{\p\p}))^{1/2}
\Del(\kappa^{\p\p}(I_1^{\p\p}))^{1/2}
\\
&c_{\Sig_{W_-}^1,\Sig^1_{V_+}}
(\kappa^{\p\p}(I^{\p\p}_2))
D^{G_{-+}}(\kappa^{\p\p}(I_2^{\p\p}))^{1/2}
\Del(\kappa^{\p\p}(I_2^{\p\p}))^{1/2}
\big\}
\end{align*}
Here, for $\kappa^{\p\p}\in \Xi(d_V,d_W)$, the pair $(I^{\p\p}_1,I^{\p\p}_2)$ lies in $\CI^+(\kappa^{\p\p})$ (resp. $\CI^-(\kappa^{\p\p})$) if and only if
\begin{itemize}
    \item
$I_{\kappa^{\p\p}} = I^{\p\p}_1\sqcup I^{\p\p}_2$;

    \item
$\kappa^{\p\p}(I^{\p\p}_1) =(I^{\p\p}_1,(F_{\pm i})_{i\in I^{\p\p}_1}, (F_i)_{i\in I^{\p\p}_1}, (u_i)_{i\in I^{\p\p}_1})\in \Xi(d_{V_+},d_{W_+})$ (resp. $\Xi(d_{V_+},d_{W_-})$) and $\kappa^{\p\p}(I_2^{\p\p})\in \Xi(d_{V_-},d_{W_-})$ (resp. $\Xi(d_{V_-},d_{W_+})$).
\end{itemize}
Thus it suffices to show that $f_{1,e_0}=f_{2,e_0}$, which follows verbatim
from \cite[Prop.~3.3]{MR3155345}. We note that Corollary
\ref{cor:germfunaslimitofquasichar} is the exact archimedean analogue of
\cite[3.1~(3)]{MR3155345}, which is used in the proof of that proposition.
\end{proof}

We now deduce Proposition \ref{pro:endoscopicid}.

We first record the special case that identifies the stable geometric
multiplicity with the stable multiplicity. Taking $s_V=1$ and $s_W=-1$ in
Proposition \ref{pro:analoguewaldspurger3.3}, we get
\[
V_+ = V_\qs, \quad V_- = 0, \quad W_+ = 0,\quad W_-=W_\qs,
\]
where $(W_\qs,V_\qs)$ is introduced in Remark \ref{rmk:uniqueqsggp}.
Thus Proposition \ref{pro:analoguewaldspurger3.3} becomes
\[
\sum_{
\substack{
\alp\in H^1(\BR,\SO(W))
\\
e(G_\alp) = e_0}
}
m_\geom(\Sig^{-1}_{W_\alp},\Sig^1_{V_\alp}) =
\frac{1}{2}
\big(
e_0
m^S_\geom(0,\Sig^1_{V_+})
m^S_\geom(\Sig^1_{W_-},0)
+m^S_\geom(\Sig^1_{W_-},\Sig^1_{V_+})
\big).
\]
Summing over $e_0=\pm 1$ and applying Theorem \ref{thm:luothesis}, we obtain
the following corollary.

\begin{cor}\label{cor:analoguewaldspurger3.3}
The following identity holds:
\[
m^S_\geom(\Sig^1_{W_\qs},\Sig^1_{V_\qs}) =
\sum_{
\alp\in H^1(\BR,\SO(W))
}
m_\geom(\Sig^{-1}_{W_\alp},\Sig^1_{V_\alp}) = m^S_{W_\qs,V_\qs}.
\]
\end{cor}

Finally, Theorem \ref{thm:luothesis} gives
\[
m(\Sig^{s_W}_{W_\alp},\Sig^{s_V}_{V_\alp})
=
m_\geom(\Sig^{s_W}_{W_\alp},\Sig^{s_V}_{V_\alp}),
\]
and the corollary above identifies the stable geometric terms with the stable
multiplicities appearing in Proposition \ref{pro:endoscopicid}. Hence
Proposition \ref{pro:endoscopicid} follows from Proposition
\ref{pro:analoguewaldspurger3.3}.

\bibliographystyle{amsalpha}
\bibliography{GGP_RealSO_Epsilon}

\end{document}